%% file: Semistable-J-9-16.tex
\documentclass[11pt,letterpaper]{amsart}

\usepackage[margin=1.5in]{geometry}
\usepackage[T1]{fontenc}
\usepackage[utf8]{inputenc}
\usepackage{lmodern}
\usepackage{microtype}
\usepackage{amsmath,amssymb,amsthm,mathtools}
\usepackage{empheq}
\usepackage{mathrsfs}
\usepackage{esint}
\usepackage{enumitem}
\usepackage{etoolbox}
\usepackage{booktabs}
\usepackage{xcolor}
\usepackage[pdfusetitle,
 bookmarks=true,bookmarksnumbered=false,bookmarksopen=true,bookmarksopenlevel=1,
 breaklinks=false,pdfborder={0 0 1},backref=false,colorlinks=false]
 {hyperref}
\usepackage[nameinlink,capitalise,noabbrev]{cleveref}
\usepackage[backend=biber,style=numeric,sorting=nyt,
  doi=false,url=false,maxbibnames=10]{biblatex}
\AtEveryBibitem{\clearfield{note}}
\allowdisplaybreaks
\setlist[itemize]{leftmargin=2em,itemsep=0.25em,topsep=0.4em}
\setlist[enumerate]{leftmargin=2.2em,itemsep=0.3em,topsep=0.4em}

\newtheorem{theorem}{Theorem}[section]
\newtheorem{proposition}[theorem]{Proposition}
\newtheorem{lemma}[theorem]{Lemma}

\theoremstyle{definition}
\newtheorem{definition}[theorem]{Definition}

\theoremstyle{remark}
\newtheorem{remark}[theorem]{Remark}
\theoremstyle{definition}
\newtheorem{example}[theorem]{Example}

\AtBeginEnvironment{proposition}{\crefalias{theorem}{proposition}}
\AtBeginEnvironment{lemma}{\crefalias{theorem}{lemma}}
\AtBeginEnvironment{corollary}{\crefalias{theorem}{corollary}}
\AtBeginEnvironment{claim}{\crefalias{theorem}{claim}}
\AtBeginEnvironment{definition}{\crefalias{theorem}{definition}}
\AtBeginEnvironment{assumption}{\crefalias{theorem}{assumption}}
\AtBeginEnvironment{remark}{\crefalias{theorem}{remark}}

\crefname{theorem}{Theorem}{Theorems}
\Crefname{theorem}{Theorem}{Theorems}
\crefname{proposition}{Proposition}{Propositions}
\Crefname{proposition}{Proposition}{Propositions}
\crefname{lemma}{Lemma}{Lemmas}
\Crefname{lemma}{Lemma}{Lemmas}
\crefname{corollary}{Corollary}{Corollaries}
\Crefname{corollary}{Corollary}{Corollaries}
\crefname{definition}{Definition}{Definitions}
\Crefname{definition}{Definition}{Definitions}
\crefname{remark}{Remark}{Remarks}
\Crefname{remark}{Remark}{Remarks}

\newcommand{\ddc}{\mathrm{dd}^c}
\newcommand{\PSH}{\operatorname{PSH}}

\newcommand{\tr}{\operatorname{tr}}

\newcommand{\id}{\mathrm{Id}}

\newcommand{\R}{\mathbb R}
\newcommand{\C}{\mathbb C}

\newcommand{\dist}{\operatorname{dist}}

\newcommand{\lsc}{\operatorname{LSC}}
\newcommand{\BarJ}{\mathfrak B_J}
\newcommand{\NullJ}{\operatorname{Null}_J}

\numberwithin{equation}{section}
\newcommand{\cpt}{\Subset}

\hypersetup{
  pdftitle={The J-Null Locus and Local Regularity of Semistable Weak Solutions of the J-Equation},
  pdfauthor={Hao Fang, Biao Ma, Jinyang Wu},
  pdfsubject={Semistable J-equation, numerical null locus, and weak solutions}
}

\title[$J$-Null Locus and Local Regularity] {The $J$-Null Locus and Local Regularity of\\ Weak Solutions of the  Semistable $J$-Equation}
\author[Hao Fang]{Hao Fang}
\address{Department of Mathematics, University of Iowa, Iowa City, IA 52246,
USA}
\email{hao-fang@uiowa.edu}
\author[Biao Ma]{Biao Ma}
\address{School of Mathematical Sciences, East China Normal University, 500
Dongchuan Road, Shanghai, China.}
\email{bma@math.ecnu.edu.cn}
\thanks{H. F.'s work is partially supported by a Simons Foundation mathematics collaboration grant. B. M. is partially supported by NSFC grant (No.~12471052).}
\begin{document}

\begin{abstract}
For the semistable \(J\)-equation, we prove that the numerical \(J\)-null locus coincides with the ambient \(C^2\)-singular locus of Murakami’s weak solution. Consequently, this singular locus is a proper analytic subset with finitely many irreducible components, and the weak solution is locally smooth outside the \(J\)-null locus.

The proof relies on two analytic ingredients. First, we establish a regularization theorem for singular \(J\)-subsolutions, showing that Demailly’s global regularization preserves quantitative strict cone conditions. Thus, singular strict subsolutions with prescribed analytic poles can be replaced by smooth strict subsolutions away from their pole sets. Second, we derive relative a priori estimates adapted to these subsolutions: a relative \(L^\infty\)-estimate from determinant control and a weighted second-order estimate yielding uniform \(C^2\)-bounds on compact subsets of the regular locus.
\end{abstract}
\maketitle
\tableofcontents

\input{j-intro-0911}

\section{Conventions, weak solutions, and subsolutions}\label{sec:setup}
We retain the definitions of \cref{sec:introduction}.
Integrals over a subvariety mean integration over its regular locus, or
equivalently the corresponding cohomological pairing on a resolution.
The notation $\omega_v$ is also used for the current associated with an
$\omega$-psh potential.
\subsection{Operators and notation }

For a Hermitian matrix $A$, the notation $A>0$ means positive definite, and $A\ge0$ means
semipositive. 

\begin{definition}
Let $A,B$ be a $n\times n$ Hermitiam matrix, and  $A>0$, $B\geq0$.  let 
\[
  Q_\gamma(A)=\tr_AB,\qquad P_B(A):=\max_{\substack{H\subset \C^{n}\\\dim_\C H=n-1}}
                  \tr_{A|_H}(B|_H).  
\]

Let $\gamma\geq0$ be a smooth positive  reference form on $X$. Let $\omega,\gamma$ be a Hermitian form on $T_x^{1,0}X$ with  $\omega>0$ and $\gamma\geq0$. We may define similarly, 
\[
Q_\gamma(A)=\tr_A\gamma, \qquad P_\gamma(\omega):=\max_{\substack{H\subset T_x^{1,0}X\\\dim_\C H=n-1}}
                  \tr_{A|_H}(\gamma|_H).
\]
\end{definition}
We extend $P_\gamma$ and
$Q_\gamma$ by $+\infty$ at singular forms.

\begin{lemma}\label{lem:F-properties}
Both $P_\gamma$ and $Q_\gamma$ are convex and decreasing under
 adding semipositive forms. Moreover,
\[
 P_\gamma(A)=\max_{\substack{B\ge0\\B\ne0}}
                  \lim_{t\to\infty}Q_\gamma(A+tB).
\]
\end{lemma}
See \cite[Lemmas~5.8 and~5.10]{FangMa24}. 

For $F\in\{P,Q\}$ and $s>0$, we will use 
$F_\gamma(sA)=s^{-1}F_\gamma(A)$ and
$F_{s\gamma}(A)=sF_\gamma(A)$. In particular, if $A>0$,
$0\le t<1$, and $B\ge-tA$, then $A+B\ge(1-t)A>0$ and
\begin{equation}\label{eq:F-perturbation}
 F_\gamma(A+B)\le\frac{F_\gamma(A)}{1-t}.
\end{equation}
The operators are also nondecreasing
in the reference form. In particular, if
$\gamma_0\le\gamma\le(1+\delta)\gamma_0$, then
\[
 F_{\gamma_0}(A)\le F_\gamma(A)
                  \le(1+\delta)F_{\gamma_0}(A).
\]
\begin{lemma}\label{lem:elementary}
    For $F\in\{P,Q\}$, let $A>0$ and $B\geq0$. If $F_B(A)\leq c$, then for $t\geq0$, \[F_B(A+tB)\leq \frac{c}{1+\frac{c}{n}t}.\]
\end{lemma}
\begin{proof}
   For $F=Q$,  put
$q_A(t)=Q_B(A+tB)$.  Cauchy--Schwarz gives
\[
 -q_A'(t)\ge\frac{q_A(t)^2}{n},
\]
and hence
\[
 q_A(t)\le
 \frac{q_A(0)}{1+tq_A(0)/n}\leq \frac{c}{1+\frac{c}{n}t}.
\]
This proves for $Q$. Now take $A_s=A+sB_1$ with $s>0$, $B_1\geq0$ and $B_1\not=0$. Denote $Q_B(A:B_1):=\lim_{s\to +\infty}Q_B(A+sB_1)$. By previous argument, we have $$q_{A_s}(t)\le
 \frac{q_{A_s}(0)}{1+tq_{A_s}(0)/n}.$$ Take $s\to+\infty$ and then maximize in $B_1$, we have \[P_{B}(A+tB)\leq \max_{0\not=B_1\geq0}\frac{Q_B(A:B_1)}{1+Q_B(A:B_1)t/n}\le \frac{P_B(A)}{1+P_B(A)t/n}.\] The last line is due to the increasing of $s\to \frac{s}{1+st/n}$. We have finished the proof.
\end{proof}
Let $T=\theta+dd^c v$ be a closed  $(1,1)$-current,
where $\theta$ is a smooth closed real $(1,1)$-form and
$v$ is $\theta$ quasi-plurisubharmonic. We use the logarithmic normalization
\[
\nu(T,x):=\nu(v,x):=
\lim_{r\downarrow0}
\frac{\sup_{B_R(x)}v-\sup_{B_r(x)}v}{\log(R/r)},
\]
where the balls are taken in local coordinates and $R>0$ is fixed
and sufficiently small. Thus $\nu(\log|z|,0)=1$, and adding a
smooth function to $v$ does not change its Lelong numbers.
For $b>0$, set
\[
E_b(T):=\{x\in X:\nu(T,x)\ge b\}.
\]

\subsection{Weak solutions and subsolutions}
Murakami shows that for a $J$-semistable pair, there exists a unique unifrom weak solution in the following sense.
\begin{definition}[Weak solutions and subsolutions]\label{def:weak-J}
A function $u\in\PSH(X,\omega)$ is a \emph{normalized weak solution} of the $J$-equation if $\sup_Xu=0$ and
\begin{align}
 c\langle\omega_u^n\rangle
 &=n\chi\wedge\langle\omega_u^{n-1}\rangle,
 \label{eq:weak-J-top}\\
 c\langle\omega_u^p\rangle
 -p\chi\wedge\langle\omega_u^{p-1}\rangle
 &\ge0,\qquad 1\le p\le n-1.
 \label{eq:weak-J-lower}
\end{align}
Here $\langle\cdot\rangle$ denotes the non-pluripolar product.

If K\"ahler current $\omega_u\in \alpha$ satisfies \eqref{eq:weak-J-lower}, then we say $\omega_u$  satisfies the \emph{cone condition} and $u$ is a \emph{subsolution} for the $J$-equation.  If $c$ is replaced by $c-\eta$ for some $\eta>0$, then $\omega_u$ satisfies the \emph{strict cone condition} and $u$ is a \emph{strict subsolution}. 
\end{definition}

Murakami showed that a positive current satisfies \eqref{eq:weak-J-lower} also satisfies the
subsolution condition in the viscosity sense and local convolution sense. To be
precise, we introduce the following definitions:

\begin{definition}\label{def:upper-test-cone}\label{def:local-convolution} Let $\gamma$ be a  reference K\"ahler form.
Let $T$ be a closed positive $(1,1)$-current on an open set $U\subset X$,
and let $d>0$. For $F\in\{P,Q\}$, the inequality
\begin{equation}\label{eq:subsolutioncondition}
 F_\gamma(T)\le d
\end{equation}
has the following formulations.
\begin{enumerate}[label=\textup{(\roman*)}]
\item \emph{Viscosity formulation.} For every local psh potential $v$
of $T$ and every upper test $q$ at a finite point $x$, one has
\[
 F_{\gamma(x)}(\ddc q(x))\le d,
\]
for any $q$ that is $C^2$ near $x$, $q\ge v$ near $x$, and $q(x)=v(x)$.
\item \emph{Local-convolution formulation.} Fix a nonnegative smooth
radial kernel $\rho$ supported in $B_1(0)$, with integral one, and put
$\rho_r(z)=r^{-2n}\rho(z/r)$. On every coordinate ball $B\subset U$,
for every constant positive Hermitian form $\gamma_0\le\gamma$ on $B$ and
every admissible convolution radius $r>0$, a local psh potential $v$
of $T$ satisfies
\[
 F_{\gamma_0}(\ddc(v*\rho_r))\le d
\]
on the smaller ball where the convolution is defined.
\item \emph{Pluripotential formulation.} For every $1\le p\le n-1$ if $F=P$, and $1\leq p \le n$ if $F=Q$,
\[
 d\langle T^p\rangle-p\gamma\wedge\langle T^{p-1}\rangle\ge0
 \quad\text{on }U.
\]
\end{enumerate}
\end{definition}

\begin{lemma}[Murakami]\label{lem:subsolution-equivalence}
For a closed positive $(1,1)$-current $T$ and $d>0$, the three
formulations of \eqref{eq:subsolutioncondition} are equivalent.
\end{lemma}
\begin{proof}
These are the inverse-trace specializations of
\cite[Propositions~2.14 and~2.34]{Murakami}. The additional $P$-condition
in the $Q$-statement there follows from  $P\le Q$.
The conclusions apply to unbounded psh potentials as well.
\end{proof}
We therefore omit the formulation when no distinction is needed.

\begin{remark}
For a K\"ahler form $\gamma$, the bound $P_\gamma(T)\le d$ implies $T\ge d^{-1}\gamma$ as currents.
Indeed, on a coordinate ball fix a constant form $0<\gamma_0\le\gamma$.
Every local convolution satisfies
$P_{\gamma_0}(T*\rho_r)\le d$, so its smallest eigenvalue relative to
$\gamma_0$ is at least $d^{-1}$. Letting $r\downarrow0$ gives
$T\ge d^{-1}\gamma_0$. Continuity of $\gamma$ and shrinking the balls
recover $T\ge d^{-1}\gamma$. In particular, a global subsolution of
\eqref{eq:subsolutioncondition} is a K\"ahler current.
\end{remark}

For a positive $(1,1)$-current $T$,  we decompose
\[
 T=T_{\mathrm{ac}}+T_{\mathrm{s}},
\]
with respect to a smooth positive volume
form,
where the absolutely continuous current $T_{\mathrm{ac}}$ is identified
with its measurable Hermitian form, and $T_{\mathrm{s}}$ is singular
with respect to the smooth volume. The next lemma offers yet another formulation of \eqref{eq:subsolutioncondition}. Its proof indicates that the nonlinear constraint is determined by the absolutely continuous part of the current.
\begin{lemma}\label{lem:ac-cone}\label{lem:mollification}\label{lem:cone-extension}
Let $T=\theta+\ddc v$ be a positive $(1,1)$-current and  $\theta$ is  smooth with local smooth potential $h$. Let $\gamma$ be a continuous
positive Hermitian form, let $a,d>0$, and let $F\in\{P,Q\}$.
\begin{enumerate}[label=\textup{(\roman*)}]
\item
$T\ge a\gamma$ and $F_\gamma(T)\le d$ are equivalent to
\begin{equation}
    \label{eq:acsubsolution}
 T_{\mathrm{ac}}\ge a\gamma,\qquad
 F_\gamma(T_{\mathrm{ac}})\le d
 \quad\text{a.e}.
\end{equation}

\item On a sufficiently small coordinate ball, let \[
u:=h+v,\qquad
v_r:=(u*\rho_r)-h.
\]
If the bounds in \textup{(i)} hold, for constant coefficient $\gamma_0$ with 
$\gamma_0\le\gamma\le(1+\delta)\gamma_0$ on the ball, then
\[\theta+\ddc v_r\ge\frac{a}{1+\delta}\gamma,\quad
 F_\gamma(\theta+\ddc v_r)\le(1+\delta)d.
\]
\item If the inequalities in \textup{(i)} holds in $X\setminus Z$ where $Z$ is a proper analytic subset, then the same inequalities hold everywhere.
\end{enumerate}
\end{lemma}
\begin{proof}
For \textup{(i)}, write
$T=\ddc u$, where $u=h+v$ is psh
and $\ddc h=\theta$. Fix a constant positive form $\gamma_0\le\gamma$.
The set
\[
 D:=\{B\ge a\gamma_0:F_{\gamma_0}(B)\le d\}
\]
is closed, convex, and stable under adding semipositive forms. Hence $D$ is
an intersection of countable supporting half-spaces
\[
 D=\bigcap_{j=1}^\infty\{A:l_j(A)\ge b_j\},
\]
where each $l_j$ is a real linear functional nonnegative on semipositive
forms.

Suppose first that $T\geq a\gamma$  and $F_\gamma(T)\leq d$ hold in the viscosity sense. They also hold with
$\gamma_0$ in place of $\gamma$. The lower bound is equivalent to
plurisubharmonicity of $u-a g_0$, where $\ddc g_0=\gamma_0$, by the
upper-test characterization of psh functions. Thus $T\ge a\gamma_0$
as currents. Applying \cref{lem:subsolution-equivalence} gives
\[
 T*\rho_r=\ddc(u*\rho_r)\in D.
\]
For every nonnegative test function $\zeta$, weak convergence of measures of bounded mass yields
\[
 \langle L_j(T),\zeta\rangle
   =\lim_{r\to0^+}\int\zeta\,L_j(T*\rho_r)\,dV
   \ge b_j\int\zeta\,dV.
\]
Thus $\mu_j:=L_j(T)-b_j\,dV$ is a positive measure. In terms of the
Lebesgue decomposition of $T$, we have
\[
 \mu_j=\bigl(L_j(T)_{\mathrm{ac}}-b_j\bigr)\,dV
             +L_j(T_{\mathrm s})\ge0.
\]
The second summand is singular with respect to $dV$. Uniqueness of the
Lebesgue decomposition therefore identifies $$(\mu_j)_{\mathrm{ac}} =\bigl(L_j(T_{\mathrm{ac}})-b_j\bigr)dV,$$ which must be positive.
Consequently $L_j(T_{\mathrm{ac}})\ge b_j$ a.e..
Intersecting the full-measure sets for the countable family of supports
gives $T_{\mathrm{ac}}\in D$ a.e..

To recover the variable reference form, for each integer $m\ge1$ choose
a countable cover by balls carrying constant forms with
$\gamma_0\le\gamma\le(1+1/m)\gamma_0$. Outside the union of the
resulting null sets,
\[
 T_{\mathrm{ac}}\ge\frac{a}{1+1/m}\gamma,\qquad
 F_\gamma(T_{\mathrm{ac}})\le(1+1/m)d.
\]
Letting $m\to\infty$ proves the almost-everywhere bounds in
\textup{(i)}.

Conversely, suppose \eqref{eq:acsubsolution} holds. For fixed $\gamma_0$ and $D$, positivity of $T_{\mathrm{s}}$ gives
\[
 L_j(T)-b_j\,dV
   =\bigl(L_j(T_{\mathrm{ac}})-b_j\bigr)\,dV
       +L_j(T_{\mathrm{s}})\ge0.
\]
Convolving these measure inequalities gives
$$L_j(T*\rho_r)\ge b_j\quad \text{for every } j,$$ and hence
$T*\rho_r\in D$. Since $\gamma_0$ is arbitrary, $P_\gamma(T)\leq d$ in the local convolution sense and hence by Lemma   \cref{lem:subsolution-equivalence}, the converse is proved.

For \textup{(ii)}, set $v_r=u*\rho_r-h$. The frozen-cone conclusion
just proved gives
\begin{align*}
 \theta+\ddc v_r&\ge a\gamma_0
       \ge\frac{a}{1+\delta}\gamma,\\
 F_\gamma(\theta+\ddc v_r)&
       \le(1+\delta)F_{\gamma_0}(\theta+\ddc v_r)
       \le(1+\delta)d.
\end{align*}
Continuity of $\gamma$ allows $\delta$ to be arbitrarily small after
shrinking the ball.

Finally, a proper analytic set has zero smooth volume. The bounds on
$T_{\mathrm{ac}}$ obtained from \textup{(i)} off $Z$ therefore hold
almost everywhere on the whole domain. Applying \textup{(i)} again
proves \textup{(iii)}.
\end{proof}

We use the regularized maximum of
\cite[Chapter~I, Lemma~5.18]{Demailly},
which underlies Richberg's gluing construction.
For $\delta>0$, let  $\widetilde {\max}_\delta$ denote the regularized maximum at level $\delta$:
\[
\widetilde{\max}_{\delta}(t_1,\ldots,t_\ell)
:=\frac{1}{\delta^\ell}\int_{\mathbb R^\ell}
\max_{1\le j\le\ell}(t_j+h_j)
\prod_{j=1}^{\ell}\theta\!\left(\frac{h_j}{\delta}\right)
\,dh_1\cdots dh_\ell.
\]
Here $\theta$ is a smooth non-negative function supported on $(-1,1)$
s.t. $\int_{\R}\theta(t)dt=1$ and $\int_{\R}t\theta(t)dt=0$.

The regularized maximum function is smooth, convex, non-decreasing in each variable.  It will be combined with Richberg's technique
\cite{Richberg1967StetigeSP} to glue local psh functions. Some related
known facts are collected in the following lemma. See \cite{Demailly}
I.5.18. 

For $F=P,Q$, the cone condition $F_\gamma(\ddc v_i)\leq d$ are preserved under taking maximum and regularized maximum.
\begin{lemma}\label{lem:cone-maxima}
Let $\gamma>0$ be a continuous form, $a,d>0$. If $v_1,\cdots,v_r$ are local potentials of $T_1,\cdots,T_r$, such that \(F_\gamma(T_i)\leq d.\) Then, \[\quad F_\gamma(\ddc \max(v_1,\cdots,v_r))\leq d.\] The same holds if max is replaced by the regularized max at level $\delta$.
\end{lemma}
\begin{proof}
For local-convalution formulation, this is proved by Chen \cite[pp. 561-562]{Chen21}. See also \cite[Lemma 8.3]{FangMa24}.  
\end{proof}

\section{Regularizing strict subsolutions}\label{sec:global-inputs}
In this section, we use Chen's mass-concentration construction
\cite{Chen21} to obtain a current in the strict cone, then apply Demailly's global
regularization and Kiselman's attenuation procedure
\cite{DemaillyFlow94} to regularize while
preserving the strict cone conditions. We also establish the
quantitative growth estimates needed for the later geometric surgery.

\subsection{Smooth approximation of the weak solution}

We first realize Murakami's weak solution as the full $L^1$-limit of smooth
solutions to perturbed $J$-equations.

\begin{theorem}\label{thm:smooth-approximants}
Assume \eqref{eq:J-semistability}.  Let $s_j\downarrow0$ and define
$\lambda_j := cs_j(1+s_j)^{n-1} \frac{\int_X\omega^n}{\int_X\chi^n}$.
There are functions $u_j\in C^\infty(X)\cap \operatorname{PSH}(X,(1+s_j)\omega)$, normalized by
$\sup_Xu_j=0$, such that $\omega_j:=(1+s_j)\omega+\ddc u_j$ satisfies
\begin{equation}\label{eq:smooth-approximation}
\left\{
\begin{aligned}
c\omega_j^n&=n\chi\wedge \omega_j^{n-1}+\lambda_j\chi^n,\\
P_\chi(\omega_j)&<c.
\end{aligned}
\right.
\end{equation}
Moreover, $u_j\to u$ in $L^1(X)$, where $u$ is the unique normalized
admissible weak solution in \cref{def:weak-J}.
\end{theorem}

\begin{proof}
For every proper irreducible subvariety $V$ of dimension $1\le p<n$,
$$
\begin{aligned}
&c(1+s_j)^p\alpha^p\cdot[V]
-p(1+s_j)^{p-1}\beta\cdot\alpha^{p-1}\cdot[V]\\
&\qquad=(1+s_j)^{p-1}
\bigl(\mathcal J_p(V)+cs_j\alpha^p\cdot[V]\bigr)>0.
\end{aligned}
$$
The definition of $\lambda_j$ gives the top-degree compatibility.
The numerical criterion proved by \cite{Chen21,Song20} gives the smooth solvability of \eqref{eq:smooth-approximation}.

For sufficiently large $j$, $\ddc u_j\ge-2\omega$, so the $L^1$-compactness of normalized 
functions in $\operatorname{PSH}(X,2\omega)$ gives a quasi-psh limit $u$. With the coefficients above, Murakami's limiting
argument shows that every cluster point satisfies
\eqref{eq:weak-J-top}--\eqref{eq:weak-J-lower}.
his admissible uniqueness theorem \cite[Theorem~1.7]{Murakami}
identifies every such limit with $u$. Hence the full sequence converges
to $u$ in $L^1(X)$.
\end{proof}

This supplies the approximating sequence used in the analytic $L^\infty$
argument.  We next construct the independent geometric current from which the
barriers are built.

\subsection{Chen's mass concentration current}
  
Chen's mass-concentration theorem in \cite{Chen21} supplies a current satisfies $P_\chi(T)<c-\eta$ in the
local-convolution sense \cref{def:local-convolution}. It provides a starting point to find other positive current with the strict cone condition.

\begin{proposition}\label{prop:mass-current}
Under \eqref{eq:J-semistability}, there are $\varepsilon_0>0$ and a closed
positive current $S\in\alpha-\varepsilon_0\beta$
such that $P_\chi(S)\le c$ in the local-convolution sense.  Consequently,
$T_c:=S+\varepsilon_0\chi\in\alpha$
is a K\"ahler current and satisfies
\begin{equation}\label{eq:Tmc-strict}
T_c>\varepsilon_0\chi,\qquad P_\chi(T_c)\le 
\frac{c}{1+\varepsilon_0c/n}=c_0<c.
\end{equation}
Moreover, for every $b>0$, the set $E_b(T_c)$ is a proper analytic subset of $X$.
\end{proposition}
The proof is essentially the same as in \cite[Theorem~1.18]{Chen21}. We include a proof for mass concentration argument for both $P,Q$ operator in the appendix \cref{prop:full-trace-concentration}.

The analyticity of $E_b(T_c)$ is Siu's theorem
\cite[Main Theorem, p.~53]{Siu}.

\begin{remark}\label{cor:strict-current}
A rescaled version may often used.  Choose
$0<\tau<1-\frac{c_{0}}{c}$
and set
$T^-:=(1-\tau)T_c\in(1-\tau)\alpha$.
Then
\begin{equation}\label{eq:Tminus-margins}
T^-\ge a_0\chi,
\qquad
P_\chi(T^-)\le c-\eta_0,
\end{equation}
where
$a_0:=(1-\tau)\varepsilon_0$, $\eta_0:=c-\frac{c_{0}}{1-\tau}>0$. 
\end{remark}




\subsection{Global regularization}
\label{sec:global-regularization}
We will use Demailly's global regularization and attenuation construction
\cite[Proposition~3.8,  Remark 4.7 and Theorem 6.1]{DemaillyFlow94} to regularize the a qpsh function.

We work on a compact K\"ahler manifold $(M^m,\gamma)$ where $\gamma$ is a  reference K\"ahler metric. We use $\gamma$ to define Demailly's modified exponential map
$$\operatorname{exph}:TM\to M,\quad (x,\zeta)\mapsto \operatorname{exph}_x(\zeta).$$
The exponential map is smooth. 
Choose the radial kernel
$$\varrho(t)=\begin{cases}
    C(1-t)^{-2}e^{-1/(1-t)}, &0\le t<1,\\
    0, &t\ge1,
\end{cases}$$
and normalized so that $\int_{\C^m}\varrho(|\zeta|^2)\,dV(\zeta)=1$, where $dV$ denotes the Lebesgue measure on $\C^m$. Notice that $\varrho$ has antiderivative $\varrho_1(t)=-Ce^{-1/(1-t)}$ for $t<1$ 

We identify $(T_xM,\gamma(x))\simeq (\C^m,|dz|^2)$ with a unitary $\gamma(x)$-frame. For a quasi-psh function $\phi$ and $r=|w|>0$, set
\begin{equation}\label{eq:global-exponential}
 \Psi_\phi(x,w):=
 \int_{T_xM}\phi\bigl(\operatorname{exph}_x(w\zeta)\bigr)
              \varrho(|\zeta|^2)\,dV(\zeta),\quad \widetilde \Psi_\phi(x,w):=\Psi_\phi(x,w)+|w|.
\end{equation}
Notice that $\widetilde\Psi_\phi$ is $S^1$-invariant in $w$. Put $\Lambda_\phi(x,r)=\partial_{\log r}\widetilde\Psi_\phi(x,r).$ By \cite[Remark~4.7]{DemaillyFlow94},
\begin{equation}\label{eq:global-radial-slope}
 \lim_{r\to 0^+}\Lambda_\phi(x,r)=\nu(\theta+\ddc\phi,x).
\end{equation} Moreover $\Lambda_\phi(x,r)$ is continuous and uniformly bounded when $r<r_0$.

Define the Demailly-Kiselman-Legendre transform 
\begin{equation}\label{eq:global-attenuation}
 \phi_{b,\varepsilon}(x)=
 \inf_{0<r<\varepsilon}
 \left\{\widetilde\Psi_\phi(x,r)
       +\frac{\varepsilon}{1-r^2/\varepsilon^2}
       -b\log(r/\varepsilon)\right\}.
\end{equation}
It is proved in \cite{DemaillyFlow94} that $\phi_{b,\varepsilon}$ is a qpsh function which is smooth on $M\setminus E_b(\phi)$, and approximates $\phi$ with a controlled Hessian lower bound. The main theorem of this section further show that the regularization can  preserve the subsolution condition. 
 
\begin{theorem}\label{thm:global-cone-regularization}
Fix $F\in\{P,Q\}$. Suppose $\sup_M\phi=0$ and, for fixed $a,d>0$,
$T=\theta+\ddc\phi\ge a\gamma$, $F_\gamma(T)\le d$. There exist constants 
$b_0,\varepsilon_0,C>0$ such that for $0<b<b_0$ and
$0<\varepsilon<\varepsilon_0$, 
$$
 T_{b,\varepsilon}:=\theta+\ddc\phi_{b,\varepsilon}\ge\frac a2\gamma,
 \qquad
 F_\gamma(T_{b,\varepsilon})\le d+C(b+\varepsilon).
$$
Moreover, $\phi_{b,\varepsilon}$ increases in $\varepsilon$ and converges to $\phi$ as $\varepsilon\to0^+$, and $T_{b,\varepsilon}$ is smooth in 
$M\setminus E_b(T)$.
\end{theorem}

We will need the following later.

\begin{lemma}\label{lem:retained-hessian}
Under the hypotheses of \cref{thm:global-cone-regularization},  
there are measurable positive matrices $A_\zeta$, smooth vector field  $v_\zeta$, $r_0>0$,
and a fixed constant $C>0$ such that, for small $|w|=r<r_0$ and $\tilde\xi:=(\xi+s\partial_w)\in T_xM\oplus T\C$, we have 
\begin{equation}      
 \label{eq:retained-lifted-hessian}
 A_\zeta \ge a(1-Cr)\id,\qquad
 F_\id(A_\zeta)\le d(1+Cr),
\end{equation} and 
\begin{align}\label{eq:lowerbound_of_hessian}
 \theta([\xi])+(\ddc\widetilde\Psi_\phi)|_{(x,w)}([\tilde\xi])
 &\ge \int_{\zeta\in T_xM}(\xi+v_\zeta s)^*A_\zeta(\xi+v_\zeta s)d\mu(\zeta)
       \\ &+\frac{|s|^2}{8 r}-C(\Lambda_\phi(x,r)+r)|\xi|^2\notag.\end{align}
Here for a vector $\xi$, we denote $[\xi]:=\sqrt{-1}\xi\wedge\bar\xi$.
\end{lemma}
\begin{proof} 
Denote $y=\operatorname{exph}_x(w\zeta)$ and write $d\mu(\zeta)=\varrho(|\zeta|^2)dV(\zeta)$. Let $h$ be the local potential of $\theta$ and let $u=h+\phi$.  Apply Demailly's  Hessian formula (\cite[Prop 3.8]{DemaillyFlow94}) (with $N=2$) to $u$ to have 
\begin{align}\label{eq:ddcPsi_phi1}
 (\ddc\widetilde\Psi_{u})|_{(x,w)}([\tilde \xi])
 &= \int_{\zeta\in T_xM} T_y ([\tau_y]+|w|^2 V_y)d\mu(\zeta)\\
 &+\frac{|s|^2}{4|w|}+ O(r) (|\xi|^2+|s|^2),\notag
\end{align}
where
\begin{align}
 \tau_y:&=\partial \operatorname{exph}_{(x,w\zeta)}(\xi^h+s\zeta^v+|w|^2\Xi^v_y),\label{eq:tau_y}
\end{align}
$\Xi$ is a smooth vector field and $V$ is a smooth $(1,1)$-vector field. The formulas in \cite[Proposition~3.8 and (3.10)]{DemaillyFlow94}, with $N=2$, give
\begin{equation}\label{eq:Xi_yandV_y}
 |\Xi|\le C|\xi|,\qquad
 |V|\le C\left(\left|\frac{\varrho_1(|\zeta|^2)}{\varrho(|\zeta|^2)}\right|+r^2\right)
 (|\xi|^2+|\xi||s|).
\end{equation}
Indeed, all terms in braces in Demailly's formula for $U$ carry the
factor $-\varrho_1/\varrho=(1-|\zeta|^2)^2$; the remaining quadratic
term and the correction $r^2\Xi\wedge\bar\Xi$ are $O(r^2|\xi|^2)$.
The constant $C$ is uniform for $0<r<r_0$.

We first identify $A_\zeta$ and $v_\zeta$. Write $\tau_y=L_\zeta\xi +s\tilde{v}_\zeta$ where $L_\zeta$ is a linear transform and $\tilde v_\zeta$ is a vector field. Then \eqref{eq:tau_y} and \cite[Prop 2.9]{DemaillyFlow94} implies that for some uniform $C>0$, $L_\zeta$ is invertible and \begin{align}(1-Cr)\mathrm{Id}\leq (L_\zeta^{-1})^*L_\zeta^{-1}\leq (1+Cr)\mathrm{Id}.\label{eq:L_zeta}\end{align}
Pick $v_\zeta:=L_\zeta^{-1}\tilde{v}_\zeta$. Then $|v_\zeta|$ is uniformly bounded. 
 By Lebesgue decomposition, $$T=T_{\mathrm{ac}}+T_{\mathrm s}.$$ View $T_{\mathrm{ac}}$ as a measurable function of positive matrix. Then if we take $A_\zeta:=L_\zeta^*T_{\mathrm{ac}}L_\zeta$,\begin{align}\label{eq:Psi_phifirstpart}
    \int T_y ([\tau_y])d\mu(\zeta)\geq \int (\xi+sv_\zeta)^*A_\zeta(\xi+sv_\zeta)d\mu(\zeta)
\end{align}
For sufficiently small $r$, $(1-Cr)\gamma\leq \mathrm{Id} \leq (1+Cr)\gamma$. By \cref{lem:ac-cone},  \begin{align}
    T_{\mathrm{ac}}\geq (1-Cr)a\mathrm{Id}, \quad F_{\mathrm{Id}}(T_{\mathrm{ac}})\leq \frac{d}{1-Cr}.\label{eq:T_acineq}
\end{align} Notice that for a positive Hermitian form $A$,  $F_{\mathrm{Id}}(L^*_\zeta AL_\zeta)=F_{(L_\zeta^{-1})^*L_\zeta^{-1}}(A).$ Therefore, by \eqref{eq:L_zeta} and \eqref{eq:T_acineq}, we have $$A_\zeta\geq (1-C'r)a\mathrm{Id},\quad F_{\mathrm{Id}}(A_\zeta)\leq d(1+C'r).$$

Now we estimate the second term in \eqref{eq:ddcPsi_phi1}. Write
$\Delta u=\sum_j u_{j\bar j}$ in normal coordinates. By \eqref{eq:Xi_yandV_y},
\begin{align*}
 r^2\int T_y(V_y)d\mu(\zeta)
 &\ge -Cr^2(|\xi|^2+|\xi||s|)
       \int(-\varrho_1(|\zeta|^2))\Delta u(y)\,dV(\zeta)\\
 &\quad-Cr^4(|\xi|^2+|\xi||s|)\int\Delta u(y)\,d\mu(\zeta).
\end{align*}
The local mass estimate \cite[(3.11)]{DemaillyFlow94} gives
$r^2\int\Delta u(y)\,d\mu(\zeta)\le C$. Integration by parts in
$\partial_{\log r}\widetilde\Psi_u$, as in
\cite[(4.5)]{DemaillyFlow94}, gives
\begin{equation}\label{eq:Lambda}
 \Lambda_u(x,r)=r+2r^2\int(-\varrho_1(|\zeta|^2))
 \Delta u(y)\,dV(\zeta)+O(r^2).
\end{equation}
The change from normal coordinates to $\operatorname{exph}_x(r\zeta)$
contributes $O(r^2)$ by the same mass estimate. Hence
\begin{equation}\label{eq:tidlePhi2}
 r^2\int T_y(V_y)d\mu(\zeta)
 \ge -C(\Lambda_u(x,r)+r)(|\xi|^2+|\xi||s|).
\end{equation}

The term contributed by the smooth form $\theta=\ddc h$ can be estimated as follows. Since $h$ is smooth,
\begin{align}
    \label{eq:theta_plus}
|\theta_x([\xi])-\ddc \Psi_h|_{(x,w)}([\tilde\xi])|\leq C(r^2|\xi|^2+r|\xi||s|+|s|^2).\end{align} Also $\Lambda_{h}(x,w)=O(r).$

By \eqref{eq:ddcPsi_phi1},  \eqref{eq:Psi_phifirstpart}, \eqref{eq:tidlePhi2}, and \eqref{eq:theta_plus},  
\begin{align}
 \theta([\xi])+&(\ddc\widetilde\Psi_{\phi})|_{(x,w)}([\tilde\xi])
 \ge \int (\xi+sv_\zeta)^*A_\zeta(\xi+sv_\zeta)d\mu(\zeta)
       \\ &+\frac{|s|^2}{4 r}-C(\Lambda_\phi(x,r)+r)(|\xi|^2+|\xi||s|)-C|s|^2\notag.\end{align}
Since $\Lambda_\phi$ is uniformly bounded, using the AM-GM inequality and choosing a small $r$, we may assume $$C(\Lambda_\phi+r)|\xi||s|+C|s|^2\leq C' r|\xi|^2+\frac{1}{8r}|s|^2.$$ We can then absorb the terms involving $|\xi||s|$ and $|s|^2$ to obtain \eqref{eq:lowerbound_of_hessian}.

\end{proof}

For an $(m+1)\times (m+1)$ Hermitian matrix $$H=\begin{pmatrix}
    B & v\\
    v^* & p
\end{pmatrix}$$ where $p>0$ is a real number, we denote the Schur complement of $H$ by $$\mathcal S(H):=B-\frac1{p}vv^*.$$ The following Lemma characterizes $\mathcal S(H)$ via minimality. See for instance, \cite[Cor 1.5.5]{Bhatia}. 

\begin{lemma} Given $H$ as above with $p>0$. For any $\xi\in\C^m$, 
    \begin{align}
 \inf_{s\in\C}
 \left\{ \begin{pmatrix}
     \xi^* & \bar s
 \end{pmatrix} H \begin{pmatrix}
     \xi \\ s
 \end{pmatrix}\right\}= \xi^* \mathcal{S}(H)\xi.
    \end{align}
\end{lemma}

Next, we prove an inequality which allows us to apply a generalized Kiselman's minimum principle following  Ross-Witt Nystr\"om  \cite{RWN} conceptually.
\begin{lemma}\label{lem:matrix-minimization}
Let $\mu$ be a probability measure, $a,d>0$, and $F\in\{P,Q\}$. Let $A_\zeta$ be a measurable
positive-definite matrix family such that $$A_\zeta\ge a\id, \quad F_\id(A_\zeta)\le d, \quad \mu\text{-a.e.}$$ Let $v_\zeta$ be a continuous vector field such that $|v_\zeta|\le L$. For $p>0$ define
$G=\int A_\zeta^{-1}\,d\mu$, $\bar v=\int v_\zeta\,d\mu$, $ B=(G+p^{-1}\bar v\bar v^*)^{-1}$.
Then
\begin{align}
 \label{eq:matrix-minimum}
 \inf_{s\in\C}
 \left\{\int(\xi+v_\zeta s)^*A_\zeta(\xi+v_\zeta s)\,d\mu
                 +p|s|^2\right\}
 \ge \xi^* B\xi,\end{align}
 Moreover, 
 \begin{align}
B\ge \frac a{(1+aL^2/ p)} \id,\quad
 F_\id( B)\le d+|\bar v|^2/p.
 \label{eq:matrix-minimum-cone}
\end{align}
\end{lemma}
\begin{proof}
$G$ is positive definite and $G\le a^{-1}\id$.
The matrix Cauchy--Schwarz gives
$$
 \int y_\zeta^*A_\zeta y_\zeta\,d\mu
 \ge\left(\int y_\zeta\,d\mu\right)^*
     G^{-1}\left(\int y_\zeta\,d\mu\right).
$$
Set $y_\zeta=\xi+v_\zeta s$. Let $$H=\begin{pmatrix}
    G^{-1}& G^{-1}\bar v\\
    \bar v^*G^{-1} & \bar v^*G^{-1}\bar v+p
\end{pmatrix}.$$ Then    $$\begin{pmatrix}
    \xi^* &\bar s
\end{pmatrix} H \begin{pmatrix}
    \xi \\s 
\end{pmatrix}\geq \xi^*\mathcal{S}(H) \xi.$$  Since $$B=(G+p^{-1}\bar v\bar v^*)^{-1}=G^{-1}-\frac{G^{-1}\bar v\bar v^*G^{-1}}{p+\bar v^*G^{-1}\bar v}=\mathcal{S}(H),$$ we have proved \eqref{eq:matrix-minimum}. The lower bound of $B$ follows from
$ B^{-1}\le(a^{-1}+L^2/ p)\id$.

Let  $\Pi$ denotes the set of all rank $m-1$ idempotent Hermitian operator. By the definition of $P$, we have $$P_{\id}(A)=\max_{K\in\Pi}\tr(KA^{-1}).$$ Since $A\mapsto \max_{K\in\Pi} \tr(KA)$ is an increasing and subadditive function,  \begin{align}
    P_{\id}(B)&=\max_{K\in \Pi}\tr(K (G+p^{-1}\bar v\bar v^*))\\
    &\leq \max_{K\in \Pi}\tr(K G)+p^{-1}|\bar v|^2 \notag\\
    &\leq \int \max_{K\in\Pi} \tr(KA_\zeta^{-1}) d\mu+p^{-1}|\bar v|^2 \notag \\
    &\leq d+p^{-1}|\bar v|^2.\notag
\end{align}
For $F=Q$, the same bound follows from the identity
$$
 Q_\id(B)=\tr(B^{-1})
 =\int Q_\id(A_\zeta)\,d\mu+p^{-1}|\bar v|^2.
$$
\end{proof}

\begin{proof}[Proof of \cref{thm:global-cone-regularization}]
Let $x\notin E_b(T)$. In a small coordinate ball of $x$ of radius $\varepsilon$, we may require that $(1-C\varepsilon)\gamma_0\leq \gamma \leq (1+C\epsilon)\gamma_0$ where $\gamma_0$ is a constant positive Hermitian form. We may use $\gamma_0$ to identify real $(1,1)$ forms with Hermitian matrices. Let $|w|=r$ small, let $$\Phi(x,w):=  \widetilde\Psi_\phi(x,r)
       +\frac{\varepsilon}{1-r^2/\varepsilon^2}
       -b\log(r/\varepsilon)$$ is strictly convex in $\log r$,  and
tends to $+\infty$ as $r\to0^+$ or $r\to \varepsilon^-$. Hence there is a unique $r_{b,\varepsilon}\in (0,\varepsilon)$ such that  $\min_{r\in(0,\varepsilon)} \Phi(x,r)$ is attained. Moreover, $r_{b,\varepsilon}(x)$  is smooth in $x$, and 
satisfies
\begin{equation}\label{eq:global-active-radius}
 \Lambda_\phi(x,r_{b,\varepsilon})
 +\frac{2\varepsilon(r_{b,\varepsilon}/\varepsilon)^2}
        {(1-(r_{b,\varepsilon}/\varepsilon)^2)^2}=b.
\end{equation}
In particular, $\Lambda_\phi(x,r_{b,\varepsilon})\le b$.

Using $\Phi_r(x,r_{b,\varepsilon}(x))=0$, we have $$\Phi_{r\bar j}=-\Phi_{rr}\partial_{\bar j}r_{b,\varepsilon}.$$ Since $\Phi(x,w)=\Phi(x,|w|)$,  differentiation of the minimum gives
$$(\phi_{b,\varepsilon})_{i\bar j} =\Phi_{i\bar j}-\frac{\Phi_{i\bar w}\Phi_{w\bar j}}{\Phi_{w\bar w}},$$
which is the Schur complement of the complex Hessian of $\Phi$ at $(x,w)$. Thus the minimum property of the Schur complement gives 
\begin{equation}\label{eq:global-schur}
 T_{b,\varepsilon}([\xi])
 =\inf_{s\in\C}\left((\theta+\ddc \Phi)[\xi+s\partial_w]\right).
\end{equation}
By  \cref{lem:retained-hessian} and \eqref{eq:matrix-minimum} with $p=\frac{1}{8r_{b,\varepsilon}}$, we have $$T_{b,\varepsilon}([\xi])\geq \xi^*(B-C_0(b+\varepsilon)\id)\xi.$$ Here $$B=(G+\bar v \bar v^*/p)^{-1},\quad G=\int A_\zeta^{-1}d\mu(\zeta),\quad \bar v=\int v_\zeta d\mu(\zeta),$$ which are given in  \cref{lem:matrix-minimization}.    

With \eqref{eq:retained-lifted-hessian} and \eqref{eq:matrix-minimum-cone}, we have $$A_\zeta\geq a(1-C\varepsilon)\id,\quad B\geq \frac{a(1-C\varepsilon)}{1+8a(1-C\varepsilon)L^2\varepsilon} \id.$$ Let $a_\varepsilon:=\frac{a(1-C\varepsilon)}{1+8a(1-C\varepsilon)L^2\varepsilon}$. Then $\lim_{\varepsilon\to0^+}a_\varepsilon=a$. For sufficiently small $\varepsilon$, we have $$T_{b,\varepsilon}\geq B-C_0(b+\varepsilon)\id\geq B\left(1-\frac{C_0(b+\varepsilon)}{a_\varepsilon}\right).$$ Then  for small $b$ and $\varepsilon$, we may also require that $C_0(b+\varepsilon)/a_\varepsilon<1/4$. Then by \cref{lem:matrix-minimization}, we have \begin{align}\label{eq:T_bepsiloninMminusE_b}F_\gamma(T_{b,\varepsilon})&\leq \frac{(1+C\varepsilon)F_{\id}(B)}{1-\frac{C_0(b+\varepsilon)}{a_\varepsilon}}\leq \frac{(1+C'\varepsilon)d+8\varepsilon L^2}{1-\frac{C_0(b+\varepsilon)}{a_\varepsilon}}\leq d+C''(b+\varepsilon),\\
T_{b,\varepsilon}&\geq (1-C\varepsilon)a_\varepsilon(1-C_0(b+\varepsilon)/a_\varepsilon)\gamma\geq \frac{a}{2}\gamma.\notag\end{align}
The constants are uniform on all of $M\setminus E_b(T)$.

Since $\phi_{b,\varepsilon}$ is quasi-psh, the singular part of
$T_{b,\varepsilon}$ is positive. As $E_b(T)$ has zero smooth
volume, the preceding lower bound yields
\[
(T_{b,\varepsilon})_{\mathrm{ac}}\geq \frac a2\gamma
\quad\text{almost everywhere}.
\]
Thus $T_{b,\varepsilon}\geq (a/2)\gamma$ globally, and
\cref{lem:ac-cone} extends the inverse-trace bound across $E_b(T)$.

Since $\phi_{b,\varepsilon}$ is quasi-psh, the singular part of
$T_{b,\varepsilon}$ is positive. As $E_b(T)$ has zero smooth
volume, the preceding lower bound gives
$(T_{b,\varepsilon})_{\mathrm{ac}}\ge (a/2)\gamma$ almost everywhere,
hence $T_{b,\varepsilon}\ge (a/2)\gamma$ globally.
By \cref{lem:ac-cone}, \eqref{eq:T_bepsiloninMminusE_b}
therefore holds on all of $M$.
\end{proof}

\subsection{Quantitative growth of a fixed regularization}
\label{sec:global-growth}
The uniform cone loss and the growth estimates have different
quantifiers. In the following proposition the smoothing parameter
$\varepsilon$ is fixed; no uniform growth assertion as
$\varepsilon\downarrow0$ is needed for surgery.

\begin{proposition}\label{prop:global-growth}
In the setting of \cref{thm:global-cone-regularization}, denote by $r_{b,\varepsilon}$ the minimizing radius for $\phi_{b,\varepsilon}$. For each
sufficiently small $b>0$, there is a proper analytic set $Z\supset E_b(T)$
such that, for each sufficiently small $\varepsilon>0$,
\begin{align}
r_{b,\varepsilon}(x)&\ge C_\varepsilon^{-1}\dist(x,Z)^L,
\label{eq:global-radius-growth}\\
L'\log\dist(x,Z)-C_\varepsilon
&\le\phi_{b,\varepsilon}(x)\le C_\varepsilon,\qquad
|\nabla\phi_{b,\varepsilon}(x)|_\gamma
\le C_\varepsilon\dist(x,Z)^{-N}
\label{eq:global-potential-growth}
\end{align}
on $M\setminus Z$, for some $L,L',N,C_\varepsilon>0$.
If $\phi$ is locally bounded off an analytic set $Y$, one may choose
$Z\subset Y$. Distances are truncated at one, with
$\dist(\,\cdot\,,\varnothing)=1$.
\end{proposition}
\begin{proof}
Denote by $\mathcal I(\phi)$ the multiplier ideal  sheaf that is the sheaf of germs of holomorphic functions
$f$ such that $|f|^2e^{-\phi}$ is locally integrable and is coherent by Nadel's theorem.  For a large integer
$k$ to be chosen below, let $Z$ be the zero locus of $\mathcal I(k\phi)$.
A theorem of Skodal (see \cite[Lemma~6.6]{DemaillyAMAG}) shows 
$$E_{2m/k}(T)\subset Z\subset E_{2/k}(T).$$
Thus $Z$ is a proper analytic set and contains $E_b(T)$ once $kb\ge2m$.
If $\phi$ is locally bounded off $Y$, then $\mathcal I(k\phi)=\mathcal O$
there, so $Z\subset Y$.

Fix finitely many holomorphic coordinate balls $B_i'\Subset B_i$  such that $\{B_i'\}$ covers $M$. On $B_i$, choose a smooth local potential $h$
of $\theta$ and shift $u=\phi+h$ so that $u\le0$.
We may choose holomorphic
functions $\sigma_1,\ldots,\sigma_N$ on $B_i$, normalized by
$\int_{B_i}|\sigma_l|^2e^{-ku}\,dz=1$, whose germs generate
$\mathcal I(k\phi)$ near $\overline{B_i'}$. Put
$$R(x):=\sum_{l=1}^N|\sigma_l(x)|^2.$$
Then $R$ is real analytic and $R^{-1}(0)=Z$ near $\overline{B_i'}$.

For fixed $A,r_0>0$, the pushforward of
$\varrho(|\zeta|^2)dV(\zeta)$ under
$\zeta\mapsto\operatorname{exph}_x(r\zeta)$ is supported in the
coordinate ball $B(x,Ar)\Subset B_i$ and has density at most $Cr^{-2m}$,
for $x\in B_i'$ and $0<r<r_0$.
Since $u\le0$ and $h$ is smooth,
$$\Psi_\phi(x,r)\ge C_0\fint_{B(x,Ar)}u-C,$$
where $C_0$ is independent of $k$ and $b$.
Subharmonicity and Jensen's inequality give
\begin{align*}
 \log|\sigma_l(x)|^2
 &\le \fint_{B(x,Ar)}\log|\sigma_l|^2\\
 &\le \log\fint_{B(x,Ar)}|\sigma_l|^2e^{-ku}
       +k\fint_{B(x,Ar)}u.
\end{align*}
Exponentiating and summing over the generators yields
$$\frac1k\log R(x)\le \fint_{B(x,Ar)}u-\frac{2m}{k}\log r+C.$$
Therefore
\begin{equation}\label{eq:exponential-weighted-Jensen}
 \Psi_\phi(x,r)\ge
 \frac{C_0}{k}\log R(x)+\frac{2mC_0}{k}\log r-C.
\end{equation}
Choose $k$ so that $kb\ge2m$ and $2mC_0/k<b/2$.

Since $\phi\le0$, taking $r=\varepsilon/2$ in
\eqref{eq:global-attenuation} gives $\phi_{b,\varepsilon}\le2\varepsilon+b$.
At $r=r_{b,\varepsilon}$, \eqref{eq:exponential-weighted-Jensen} gives
$$
 C\ge\phi_{b,\varepsilon}(x)\ge
 \frac{C_0}{k}\log R(x)
 +\left(\frac{2mC_0}{k}-b\right)\log r_{b,\varepsilon}
 +b\log\varepsilon-C.
$$
For $0<\varepsilon<\min\{r_0,1\}$, this implies, after changing $C$,
\begin{align*}
 r_{b,\varepsilon}(x)&\ge C^{-1}\varepsilon^2R(x)^{2C_0/(kb)},\\
 \phi_{b,\varepsilon}(x)&\ge
 \frac{C_0}{k}\log R(x)-b|\log\varepsilon|-C.
\end{align*}
The \L ojasiewicz inequality for $R$ gives
\eqref{eq:global-radius-growth} and the logarithmic lower bound in
\eqref{eq:global-potential-growth} on $B_i'\setminus Z$.

Write \eqref{eq:global-exponential} as
$\Psi_\phi(x,r)=\int K_r(x,y)\phi(y)\,dV_\gamma(y)$.
Its smooth kernel satisfies $|\nabla_xK_r|\le Cr^{-2m-1}$, so
$|\nabla_x\Psi_\phi(x,r)|\le Cr^{-2m-1}\|\phi\|_{L^1}$.
At the minimizing radius,
$\nabla\phi_{b,\varepsilon}=\nabla_x\Psi_\phi(x,r_{b,\varepsilon})$.
The radius bound proves the gradient assertion. Taking maxima of the
constants over the finite cover completes the proof.
\end{proof}

\section{The minimal \texorpdfstring{$J$}{J}-barrier locus}
\label{sec:minimal-barrier}\label{sec:global-barriers}

We use  regularization results of \cref{sec:global-inputs} to construct
strict subsolutions with analytic pole sets, then realize the intersection of their
pole sets by a barrier with controlled growth.

\begin{definition}[$J$-barrier]\label{def:J-barrier}
Let $\BarJ=\BarJ(X,\omega,\chi)$ consist of tuples
$\mathbf b=(\underline v,\tau,\eta,Z)$, where
$0<\tau<1$, $0<\eta<c$, and $\underline v\le0$ is quasi-psh,
continuous on $X\setminus Z$, with
$Z=\{\underline v=-\infty\}$ a proper analytic subset of $X$,
such that
$T:=(1-\tau)\omega+\ddc\underline v$
is a closed positive current satisfying
\begin{equation}\label{eq:barrier-P}
P_\chi(T)\le c-\eta.
\end{equation}
We call $\mathbf b$, or $\underline v$ when the remaining data
are understood, a \emph{$J$-barrier}.
\end{definition}

\begin{theorem}
\label{thm:strict-barrier} For $J$-semistable $(\alpha,\beta)$, let $T_c$ be the current in \cref{prop:mass-current}.
 Then for every sufficiently small $b>0$, there is a normalized barrier
$(\underline v_b,\tau,\eta,E_b(T_{c}))\in\BarJ$, smooth off its pole
set, with $\tau,\eta$ independent of $b$.
\end{theorem}

\begin{proof} 
Write $T^-=(1-\tau)\omega+\ddc\phi^-$ as in
\cref{cor:strict-current}, with $\sup_X\phi^-=0$, and put
$Z_b=E_b(T_c)=E_{(1-\tau)b}(T^-)$.
For sufficiently small $b,\varepsilon>0$,
\cref{thm:global-cone-regularization} gives
$\phi_1=(\phi^-)_{(1-\tau)b,\varepsilon}$ which is smooth off $Z_b$, with
\[
P_\chi\bigl((1-\tau)\omega+\ddc \phi_1\bigr)\le c-\eta_0/2.
\]

For any proper analytic set $Z$, Demailly's construction
\cite[Chapter~IX, Lemma~2.12, p.~413]{Demailly} provides a quasi-psh
function $\psi_Z\le0$, smooth off $Z$, with exact logarithmic poles
there and $$\ddc\psi_Z\ge-C_Z\chi.$$ Take $\psi_Z=0$ if $Z=\varnothing$.
Set
\[
\underline v_b=\phi_1+\delta\psi_{Z_b}
-\sup_X(\phi_1+\delta\psi_{Z_b}).
\]
For sufficiently small $\delta>0$, \eqref{eq:F-perturbation} and
\cref{lem:cone-extension} give \eqref{eq:barrier-P} with
$\eta=\eta_0/4$. 
Since $\phi_1$ is bounded above, $\underline v_b$ tends to $-\infty$
along $Z_b$; smoothness off $Z_b$ gives the exact pole set.
\end{proof}

Define the minimal $J$-barrier locus by
\begin{equation}\label{eq:minimal-barrier-locus-definition}
Z_J=\bigcap_{\mathbf b\in\BarJ}Z_{\mathbf b}.
\end{equation}
By \cref{thm:strict-barrier}, $Z_J\subset E_b(T_c)$ for all
sufficiently small $b>0$.

We show that $Z_J$ is realized as the polar set of a particular barrier with controlled growth near $Z_J$. 
\begin{theorem}
\label{prop:finite-realization}\label{prop:finite-barrier-maximum}
\label{thm:minimal-barrier}\label{thm:controlled-realizing-barrier}
The locus $Z_J$ is a proper analytic subset and is the exact pole set
of a normalized barrier $\underline v_*$, smooth on $X\setminus Z_J$,
such that, for some $L,\ell,C,N>0$,
\begin{align}
L\log\dist(x,Z_J)-C
&\le \underline v_*(x)\le \ell\log\dist(x,Z_J)+C,
\label{eq:two-sided-log-barrier}\\
\operatorname{Lip}_{\mathrm{loc}}\underline v_*
&\le C\dist(\,\cdot\,,Z_J)^{-N}
\quad\text{on }X\setminus Z_J.
\label{eq:polynomial-Lipschitz-barrier}
\end{align}
Distances are truncated at one, with $\dist(\,\cdot\,,\varnothing)=1$.
\end{theorem}


\begin{proof}
The descending-chain condition for analytic subsets of the compact
manifold $X$ gives finitely many barriers $\mathbf b_i=(\underline{v}_i,\tau_i,\eta_i,Z_i)$, $i=1,\cdots,m$
with $Z_J=\bigcap_iZ_{i}$. Thus $Z_J$ is proper and analytic.
Put $\tau_0=\min_i\tau_i$ and $\eta_0=\min_i\eta_i$. By \cref{lem:cone-maxima},
the normalized maximum
\[
\underline v_0=\max_i\underline v_i-\sup_X\max_i\underline v_i
\]
satisfies \eqref{eq:barrier-P}  with $\tau_0,\eta_0$ and has exact
pole set $Z_J$. It is continuous off $Z_J$. 

Apply \cref{thm:global-cone-regularization,prop:global-growth} to
$T_0=(1-\tau_0)\omega+\ddc\underline v_0$, choosing a small $b>0$
and then one sufficiently small $\varepsilon>0$.
The resulting potential $v'$ is smooth on
$X\setminus E_b(T_0)$ and satisfies
\eqref{eq:barrier-P} with $\tau_0$ and some
$0<\eta_1<\eta_0$.
By \cref{prop:global-growth}, there is a proper analytic set
\[
E_b(T_0)\subset Z'\subset Z_J
\]
such that $v'$ has a logarithmic lower bound, a global upper
bound, and a polynomial gradient bound on $X\setminus Z'$.

Similar to the proof of \cref{thm:strict-barrier}, we may find qpsh function $ \psi_{Z'}$ smooth outside $Z'$ and has logarithmic poles along $Z'$. Take $\underline v_*= v'+\delta \psi_{Z'}$ for some sufficiently small $\delta$. Then,  the resulting barrier $\underline v_*$ has exact pole set $Z'$ satisfies \eqref{eq:barrier-P} for some $0<\tau_*<\tau_0,0<\eta_*<\eta_1$ and the
estimates for $v'$ and $\psi_{Z'}$ give
\eqref{eq:two-sided-log-barrier}--\eqref{eq:polynomial-Lipschitz-barrier}.
Then $(\underline v_*,\tau_*,\eta_*,Z')$ in $\BarJ$ gives $Z_J\subset Z'$, hence $Z'=Z_J$.
\end{proof}

\section{Identification of the null locus}
\label{sec:component-removal}
\input{IdentificationofNulllocus}

\section{Relative estimates}
\label{sec:determinantal-transfer}
In this section, we prove \cref{thm:main}. We first establish relative $L^\infty$ and $C^2$ estimates for approximates solutions in \cref{thm:smooth-approximants} with respect to a barrier. By choosing a barrier smooth outside $\NullJ$, these estimates transfer to $C^\infty(X\setminus\NullJ)$ local regularity of Murakami's weak solution.

Assume throughout this section that $(\alpha,\beta)$ is $J$-semistable.
Let $u$ be the normalized admissible weak solution of
\cref{def:weak-J}, let $u_j$ be the normalized smooth approximants of
\cref{thm:smooth-approximants}, and write
\[
 \omega_j=(1+s_j)\omega+\ddc u_j.
\]
Fix an arbitrary barrier
\[
 \mathbf b=(\underline v,\tau,\eta,Z_{\mathbf b})\in\BarJ.
\]

\begin{theorem}
\label{thm:abstract-barrier-criterion}
There are
$j_{\mathbf b}\in\mathbb{N}$ and $C_{\mathbf b}>0$ such that
\[
 u_j\ge\underline v-C_{\mathbf b}
 \quad\text{on }X\setminus Z_{\mathbf b},
 \qquad j\ge j_{\mathbf b}.
\]
Consequently,
$u\in L^{\infty}_{\mathrm{loc}}(X\setminus Z_{\mathbf b})$.
\end{theorem}

The relative $L^\infty$ estimate for degenerate fully nonlinear elliptic equations has been studied in  \cite{FMWRelative}. It uses the auxiliary Monge--Amp\`ere
method of Guo--Phong--Tong \cite{GuoPhongTong23}. Similar estimates was obtained by Sui--Sun \cite{SuiSun23} for degenerate complex Hessian quotient equations. 

\begin{theorem}\label{thm:relative}
Assume in addition that $\underline v$ is smooth on
$X\setminus Z_{\mathbf b}$.  There are constants $K,C>0$, independent of
$j$, such that, for all sufficiently large $j$,
\begin{equation}\label{eq:weighted-bound}
 \tr_\chi\omega_j
 \le C\exp\bigl(K(u_j-\underline v)\bigr)
 \le C e^{-K\underline v}
 \quad\text{on }X\setminus Z_{\mathbf b}.
\end{equation}
$K$ depends only on $n,c,\eta$ and a curvature bound for
$\chi$; the constant $C$ may also depend on the barrier, but neither
constant depends on $j$.
\end{theorem}

The relative $C^2$ estimates combines
the estimates established in 
\cite{SongWeinkove,Chen21} with the systematic  $\mathcal{C}$-subsolution method of Guan \cite{Guan14} and
Sz\'ekelyhidi \cite{Szekelyhidi18}.

\subsection{Barrier as $\mathcal D$-subsolution}

For the perturbed equation \eqref{eq:smooth-approximation}, define
\[
G_j(A):=\tr_A\chi+\frac{\lambda_j}{\det_\chi A},
\qquad \det_\chi A:=\frac{A^n}{\chi^n}.
\]
Both terms decrease when the positive Hermitian form $A$ increases.  Set
\[
 \mathcal A_j:=\{A>0:P_\chi(A)<c\}.
\]
In the terminology of \cite{FMWRelative}, $(c-G_j,\mathcal A_j)$ is an
admissible pair.
 
\begin{lemma}\label{lem:det-coercivity}
Let $B>0$ be a Hermitian form satisfying
$P_\chi(B)\le c-\eta_1$ for some $0<\eta_1<c$.  There is $j_0$ such that,
for every $j\ge j_0$ and every Hermitian form $R\geq0$,
putting $A:=B+R$, one has
\begin{equation}\label{eq:det-coercivity}
G_j(A)\ge c
\quad\Longrightarrow\quad
0\le R\le A\le \frac{2}{\eta_1}\chi.
\end{equation}
Consequently, $\det_\chi R\le(2/\eta_1)^n$.
\end{lemma}

\begin{proof}
Put $A=B+R$, and let
$\mu_1\ge\cdots\ge\mu_n>0$
be the eigenvalues of $\chi$ relative to $A$.  Since $A\ge B$, inverse-matrix
order gives
\[
\sum_{k=1}^{n-1}\mu_k
=P_\chi(A)\le P_\chi(B)\le c-\eta_1.
\]
If $G_j(A)\ge c$, then
\[
\eta_1
\le \mu_n+\lambda_j\prod_{k=1}^n\mu_k
\le \mu_n\bigl(1+\lambda_j(c-\eta_1)^{n-1}\bigr).
\]
Choose $j_0$ so that
$\lambda_j(c-\eta_1)^{n-1}\le1$ for $j\ge j_0$.  Then
$\mu_n\ge\eta_1/2$, or equivalently
$A\le2\eta_1^{-1}\chi$.  Since $0\le R\le A$, this proves
\eqref{eq:det-coercivity}. 
\end{proof}

Apply \cref{lem:det-coercivity} with $\eta_1=\eta$, and set
\[
 f:=\log\left(\left(\frac{2}{\eta}\right)^n\right)
   =n\log\frac{2}{\eta}.
\]

In \cref{thm:smooth-approximants},  the equation gives
$G_j(\omega_j)=c$.  Define
\[
W_j:=u_j-\underline v
\quad\text{on }X\setminus Z_{\mathbf b},
\qquad
W_j:=+\infty
\quad\text{on }Z_{\mathbf b}.
\]
Then $W_j$ is $\lsc$ on $X$, continuous on
$X\setminus Z_{\mathbf b}$, and
$W_j(x)\to+\infty$ as $x\to Z_{\mathbf b}$.
\begin{proposition}\label{prop:det-transfer}
For every $j\ge j_0$, every $x\in X\setminus Z_{\mathbf b}$, and every $C^2$ lower test
$q$ for $W_j$ at $x$, put
$A_q:=\tau\omega(x)+\ddc q(x)$.
If $A_q\ge0$, then
\begin{equation}\label{eq:det-test-bound}
\det_\chi A_q\le e^f.
\end{equation}
The constant is independent of $j$.
\end{proposition}

\begin{proof}
If $q$ touches $W_j$ from below at $x$, then $u_j-q$ touches
$\underline v$ from above at the same point.  The upper-test formulation
of \eqref{eq:barrier-P} therefore gives
\[
B:=(1-\tau)\omega(x)+\ddc(u_j-q)(x),
\qquad B>0,\qquad P_{\chi(x)}(B)\le c-\eta.
\]
Set $H:=A_q+s_j\omega(x)\ge0$.  The backgrounds give the exact identity
$B+H=\omega_j(x)$.
Since $G_j(B+H)=G_j(\omega_j(x))=c$, \cref{lem:det-coercivity} gives
$\det_\chi H\le e^f$.
Since $0\le A_q\le H$, determinant monotonicity gives
\eqref{eq:det-test-bound}.
\end{proof}

In the terminology of \cite{FMWRelative}, the proposition says that
\[
 \bigl((1-\tau)\omega+\ddc\underline v,\tau\omega,f\bigr)
\]
is a $\mathcal D$-subsolution for the admissible pair
$(c-G_j,\mathcal A_j)$.

\subsection{The relative $L^\infty$ estimate}
\label{sec:punctured-estimate}\label{sec:weak-limit}

We use the following specialization of
\cite[Theorem~3.4]{FMWRelative}.

\begin{theorem}\label{thm:punctured}
Let $(X,\chi)$ be a compact K\"ahler manifold of dimension $n$, let
$\omega$ be a K\"ahler form, let $\tau>0$, and put
\(
V_\tau:=\int_X(\tau\omega)^n.
\)
Let $Z\subsetneq X$ be closed. Let $U$ be a lower semicontinuous function on $X$, continuous on $X\setminus Z$ with $U^{-1}(+\infty)=Z$.
Let $f\in C^0(X)$.  Assume that every $C^2$ lower test $q$ for $U$ on
$X\setminus Z$ satisfies
\[
\tau\omega+\ddc q\ge0
\quad\Longrightarrow\quad
\det_\chi(\tau\omega+\ddc q)\le e^f.
\]
Suppose that, for some $p>n$ and $H,L<\infty$,
\[
\operatorname{Ent}_p:=\frac1{V_\tau}
\int_X
\left[1+(f-\log V_\tau)_+\right]^p e^f\chi^n\le H,
\qquad
\int_X(-U)_+\chi^n\le L.
\]
Then $\sup_X(-U)\le C(X,\chi,\omega,\tau,p,H,L)$.
\end{theorem}

We apply the theorem to $W_j=u_j-\underline v$.  Its constants are
uniform in $j$.

\begin{proof}[Proof of \cref{thm:abstract-barrier-criterion}]
Apply \cref{thm:punctured} to $U=W_j$, $Z=Z_{\mathbf b}$, and the
constant function $f=\log((2/\eta)^n)$.  The lower-test bound is
\cref{prop:det-transfer}.  For any fixed $p>n$, the entropy term equals
\[
\operatorname{Ent}_p=\frac{e^f\int_X\chi^n}{V_\tau}
\left[1+(f-\log V_\tau)_+\right]^p.
\]
Since
$\underline v\le0$ and $\sup_Xu_j=0$,
$(-W_j)_+ =(\underline v-u_j)_+ \le -u_j$.
After discarding finitely many indices, $s_j\le1$, and
\eqref{eq:smooth-approximation} gives
$\ddc u_j\ge-2\omega$.
The normalized family of $2\omega$-psh functions has a uniform $L^1$ bound
by \cite[Proposition~2.7]{GZ05}.
Hence
\[
 \int_X(-W_j)_+\chi^n \le \int_X(-u_j)\chi^n \le L_0,
\]
with $L_0$ independent of $j$.  The punctured estimate yields
\begin{equation}
    \sup_X(-W_j) = \sup_{X\setminus Z_{\mathbf b}}(\underline v-u_j) \le C_{\mathbf b}\label{eq:relative-lower}.
\end{equation}

It remains to pass the bound to the $L^1$-limit.  Let
$K\Subset X\setminus Z_{\mathbf b}$, and choose an open set
$K\Subset V\Subset X\setminus Z_{\mathbf b}$.  By
\cref{thm:smooth-approximants}, $u_j\to u$ in $L^1(X)$.  Since
$\underline v$ is continuous on $V$, \eqref{eq:relative-lower} holds for
$u$ almost everywhere on $V$, and therefore everywhere for its upper
semicontinuous representative.  Since $u\le0$ by normalization,
$\inf_V\underline v-C_{\mathbf b}\le u\le0$ on $K$.
As $K$ was arbitrary,
\[
 \underline v-u\le C_{\mathbf b}
 \quad\text{on }X\setminus Z_{\mathbf b}.
\]
Thus $u\in L^\infty_{\mathrm{loc}}(X\setminus Z_{\mathbf b})$.
\end{proof}

\subsection{The relative   $C^2$   estimate}
\label{sec:weighted-regularity}
\label{sec:trace}
\label{sec:coercivity}
\label{sec:maximum}
Throughout this subsection and the next, assume that
$\underline v$ is smooth on $X\setminus Z_{\mathbf b}$.

For the solution $\omega_j$, define
\[
 \mathcal L_j h:=-D G_j|_{\omega_j}(\ddc h),
 \qquad w_j:=\tr_\chi\omega_j.
\]
Thus $\mathcal L_j$ is elliptic, and
$G_j(\omega_j)=c$. 

To establish $C^2$ estimates, we need the following 2 lemmas.
\begin{lemma}\label{lem:log-trace}
There is a constant
\(
 C_0=C_0\bigl(n,c,\|\operatorname{Rm}(\chi)\|_{C^0(\chi)}\bigr)>0
\)
such that
\begin{equation}\label{eq:log-trace}
 \mathcal L_j\log w_j\ge-C_0
\end{equation}
for every $j$.  In particular, $C_0$ is independent of $s_j$,
$\lambda_j$, $u_j$, and the barrier.
\end{lemma}
\begin{proof}
It follows from \cite[Proposition~2.1]{Chen21}. 
\end{proof}

The following lemma is essentially established in \cite{SongWeinkove,Chen21}. The formulation is inspired by \cite[Lemma 4.1]{Guan14} and \cite[Proposition 6]{Szekelyhidi18}. 
\begin{lemma}\label{lem:coercivity}
Fix $0<\eta<c$ and $\Lambda>0$.  There are constants $N,\delta>0$,
depending only on $n,c,\eta,\Lambda$, with the following property.  For
every $j$ such that $0\le\lambda_j\le\Lambda$, if $A,B>0$ satisfy
\[
 G_j(A)=c,\qquad P_\chi(B)\le c-\eta,\qquad \tr_\chi A\geq N,
\]
then 
\begin{equation}\label{eq:coercivity}
 -D G_j|_A(B-A)
 \ge\delta\bigl(1-D G_j|_A(\chi)\bigr).
\end{equation}
\end{lemma}

\begin{proof}
Let $B_M\le B$ be obtained by replacing each $\chi$-eigenvalue $b_i$ of
$B$ by $\min\{b_i,M\}$.  For $M=2(n-1)/\eta$, the pairs
$(B_M,\lambda_j)$, $0\le\lambda_j\le\Lambda$, form a compact family and
satisfy $$P_\chi(B_M)\le P_\chi(B)+(n-1)/M\le c-\eta/2.$$
Apply the proof of \cite[Proposition~7.10]{FangMa24} uniformly to this
family; positivity of $-D G_j|_A$ gives
\eqref{eq:coercivity} for $B$.
\end{proof}

\begin{proposition}\label{prop:metric}
Let $C_0$ be as in \cref{lem:log-trace}, and let $N,\delta$ be as in
\cref{lem:coercivity} with $\Lambda=1$.  For any $K>C_0/\delta$ and all
sufficiently large $j$,
\[
 \tr_\chi\omega_j
 \le N e^{KC_{\mathbf b}}
       \exp\bigl(K(u_j-\underline v)\bigr).
\]
\end{proposition}

\begin{proof}
We may assume $\lambda_j\le1$ and the estimate in
\cref{thm:abstract-barrier-criterion} holds.  On
$X\setminus Z_{\mathbf b}$, set
\[
 \Phi_j:=\log w_j-K(u_j-\underline v).
\]
Since $\underline v$ tends to $-\infty$ along $Z_{\mathbf b}$,
$\Phi_j$ attains its maximum at a point
$x_j\in X\setminus Z_{\mathbf b}$.  If $Z_{\mathbf b}=\varnothing$, this
follows directly from compactness.

At $x_j$, ellipticity and \cref{lem:log-trace} give
\[
 \mathcal L_j(u_j-\underline v)
 \ge K^{-1}\mathcal L_j\log w_j
 \ge-\frac{C_0}{K}.
\]
Put
\(
 B:=(1-\tau)\omega+\ddc\underline v.
\)
The barrier condition gives $P_\chi(B)\le c-\eta$.  Since
\[
 B-\omega_j=-(\tau+s_j)\omega-\ddc(u_j-\underline v),
\]
we obtain at $x_j$
\[
 -D G_j|_{\omega_j}(B-\omega_j)
 \le \frac{C_0}{K}<\delta.
\]
As $-D G_j|_{\omega_j}(\chi)\ge0$,
\cref{lem:coercivity} implies $w_j(x_j)<N$.  The relative lower bound
from \cref{thm:abstract-barrier-criterion} gives
\[
 \Phi_j(x_j)\le\log N+KC_{\mathbf b}.
\]
The maximality of $\Phi_j$ proves the proposition.
\end{proof}

\begin{proof}[Proof of \cref{thm:relative}]
The lower metric bound was proved above.
\Cref{prop:metric} gives the first inequality in
\eqref{eq:weighted-bound}; the second follows from $u_j\le0$.
\end{proof}

\subsection{Smooth convergence and proof of the main theorem}
\label{sec:regularity}

Fix relatively compact open sets
$L\cpt U'\cpt U\cpt X\setminus Z_{\mathbf b}$.  The relative bound of
\cref{thm:abstract-barrier-criterion}, the normalization $u_j\le0$, and
$Q_\chi(\omega_j)<c$ give
\begin{equation}\label{eq:local-range}
 \|u_j\|_{L^\infty(U)}\le C_U,
 \qquad
 c^{-1}\chi\le\omega_j\le C_U\chi
 \quad\text{on }U.
\end{equation}
$C_U$ can be chosen independent of $j$. On the compact matrix range in \eqref{eq:local-range}, the operators
$-G_j$ are uniformly elliptic and concave.  Their dependence on
$x$ is uniformly smooth because $\omega$ and $\chi$ are fixed and
$s_j,\lambda_j$ are bounded.

\begin{proposition}\label{prop:local-smoothness}
For every compact $L\cpt X\setminus Z_{\mathbf b}$ and every integer
$m\ge0$,
\begin{equation}\label{eq:all-local-bounds}
 \sup_j\|u_j\|_{C^m(L)}<\infty.
\end{equation}
Moreover,
\begin{equation}\label{eq:smooth-limit}
 u_j\longrightarrow u
 \quad\text{in }C^\infty_{\mathrm{loc}}(X\setminus Z_{\mathbf b}),
\end{equation}
and on this complement
\[
 \qquad Q_\chi(\omega_u)=c.
\]
\end{proposition}

\begin{proof}
The complex Evans--Krylov estimate
\cite[Theorem~1.2]{TWWY}, applied to $-G_j$ on $U$, gives an
$\alpha\in(0,1)$ such that
\begin{equation}\label{eq:evans-krylov}
 \|u_j\|_{C^{2,\alpha}(U')}\le C_{U'}.
\end{equation}
Since $G_j$ is now uniformly elliptic in $U$, the  standard interior Schauder estimates gives
\eqref{eq:all-local-bounds}.

By \cref{thm:smooth-approximants}, $u_j\to u$ in $L^1(X)$.  Every
subsequence therefore has a further subsequence converging smoothly on
compact subsets of $X\setminus Z_{\mathbf b}$, and its limit agrees with $u$.  Thus the full sequence
converges as in \eqref{eq:smooth-limit}.

Finally, 
\[
0\le\lambda_j\det_{\omega_j}\chi\le\lambda_jc^n\longrightarrow0.
\]
Since $s_j\to0$, smooth local convergence in
\eqref{eq:smooth-approximation} yields $Q_\chi(\omega_u)=c$.
\end{proof}

\begin{proof}[Proof of \cref{thm:main}]
By \cref{thm:locus-identification},
$Z_J=\NullJ(\alpha,\beta)$, and this set is a proper analytic subset with
finitely many irreducible components.  Choose the realizing barrier of
\cref{thm:controlled-realizing-barrier}; its pole set is $Z_J$ and its
potential is smooth on $X\setminus Z_J$.  Then
\cref{thm:abstract-barrier-criterion,thm:relative,prop:local-smoothness}
shows that $u$ is smooth on
$X\setminus Z_J=X\setminus\NullJ(\alpha,\beta)$.
\end{proof}

\section{Comparison and examples}\label{sec:comparison-examples}
\subsection{The shifted non-K\"ahler locus}
\label{sec:shifted-class}

Having proved \cref{thm:main}, we compare $Z_J$ with the non-K\"ahler locus of
the shifted class naturally associated with the pointwise $J$-cone:
$\Theta_J:=c\alpha-(n-1)\beta\in H^{1,1}(X,\R)$.
Recall that the non-K\"ahler locus of a big class $\Theta$ is
$E_{\mathrm{nK}}(\Theta) :=\bigcap_R E_+(R)$,
where $R$ ranges over the K\"ahler currents in $\Theta$.  Boucksom proves
that this intersection is the positive-Lelong locus of one K\"ahler current
with analytic singularities
\cite[Definition~3.16 and Theorem~3.17(ii)]{Boucksom}.

\begin{proposition}\label{prop:shifted-upper}
If $\Theta_J$ is big, then
\begin{equation}\label{eq:shifted-upper}
 \NullJ(\alpha,\beta)\subset E_{\mathrm{nK}}(\Theta_J).
\end{equation}
\end{proposition}

\begin{proof}
Choose, by Boucksom's theorem, a K\"ahler current $R\in\Theta_J$ with
analytic singularities such that
$E_+(R)=E_{\mathrm{nK}}(\Theta_J)$.
Write
$R=c\omega-(n-1)\chi+\ddc\psi$,
where $\psi\le0$ is smooth off $E_+(R)$ and has that set as its exact polar
set.  Since $R$ is a K\"ahler current, choose $\delta>0$ such that
$R\ge\delta\chi$.  Choose $\tau>0$ so small that
\begin{equation}\label{eq:shifted-tau-choice}
 (1-\tau)(n-1+\delta)>n-1,
\end{equation}
and set $\underline v=(1-\tau)\psi/c$.  Then
\begin{align*}
 T_{\mathbf b}
 &:=(1-\tau)\omega+\ddc\underline v
 =\frac{1-\tau}{c}\bigl((n-1)\chi+R\bigr),
 \\
 T_{\mathbf b}
 &\ge a\chi,
 \qquad
 a:=\frac{(1-\tau)(n-1+\delta)}{c}.
\end{align*}
On the complement of $E_+(R)$ this current is smooth, and the elementary
eigenvalue bound $T_{\mathbf b}\ge a\chi$ gives
$P_\chi(T_{\mathbf b}) \le\frac{n-1}{a} =c-\eta$, $\eta:=c-\frac{n-1}{a}>0$
by \eqref{eq:shifted-tau-choice}. The potential is already smooth
off its exact analytic pole set, so \cref{def:J-barrier} applies. Thus
$\bigl(\underline v,\tau,\eta,E_{\mathrm{nK}}(\Theta_J)\bigr)\in\BarJ$.
The defining intersection $Z_J$ of all barrier polar sets is contained in
this one. Now \cref{thm:locus-identification} gives \eqref{eq:shifted-upper}.
\end{proof}

\subsection{An example of mixed-dimensional null locus}
\begin{example}[Mixed-dimensional null locus and bounded potentials]
\label{ex:J-mixed-null-components}
Let $1\le m\le n-2$, let $Y$ be a smooth connected projective $m$-fold,
and choose a very ample line bundle $L$ on $Y$. Put $r=n-m$ and, in the
lines convention, set
\[
 X_0=\mathbb P_Y\bigl(\mathcal O_Y\oplus(L^{-1})^{\oplus r}\bigr),
 \qquad Y_0=\mathbb P_Y(\mathcal O_Y).
\]
Write $h=\pi^*c_1(L)$ and $\xi=c_1(\mathcal O_{X_0}(1))$.
The semiample class $\xi$ contracts exactly $Y_0$ and embeds its
complement, while $|\xi-h|$ has base locus $Y_0$.
For sufficiently small $a>0$, define
\begin{equation}\label{eq:J-mixed-base-classes}
 A=\xi+ah,\qquad
 b=\frac mn-\frac{ra\,h\cdot A^{n-1}}{n\,\xi\cdot A^{n-1}},
 \qquad B=b\xi+ah.
\end{equation}
Then $A,B$ are K\"ahler, $b\to m/n$ as $a\downarrow0$, and
$nBA^{n-1}/A^n=m$. Their numerical defects are
\[
 \mathcal J_p^0(V)
 =\bigl((m-pb)\xi+(m-p)ah\bigr)A^{p-1}\cdot V.
\]
For $V\subset Y_0$, this equals $(m-p)a^p\int_Vc_1(L)^p$ and vanishes
only for $V=Y_0$. For $V\not\subset Y_0$, a divisor in $|\xi-h|$
not containing $V$ gives $hA^{p-1}\cdot V\le\xi A^{p-1}\cdot V$.
Hence, after choosing $a$ small enough,
\begin{equation}\label{eq:J-mixed-positive-margin}
 \mathcal J_p^0(V)\ge\kappa_p\,\xi A^{p-1}\cdot V>0,
 \qquad
 \kappa_p=m-pb-\max\{p-m,0\}a>0
 \quad(1\le p<n).
\end{equation}

Blow up $q\in X_0\setminus Y_0$, writing
$f:X\to X_0$, $E\cong\mathbb P^{n-1}$, and $Y'$ for the strict
transform of $Y_0$. For small $\varepsilon>0$, set
\begin{equation}\label{eq:J-mixed-blowup-classes}
 \begin{gathered}
 t_\varepsilon=\frac{m\varepsilon^n}
 {n(n-1)\xi\cdot A^{n-1}},\qquad
 \alpha_\varepsilon=f^*A-\varepsilon[E],\\
 \beta_\varepsilon=f^*(B+t_\varepsilon\xi)
                 -\frac{m\varepsilon}{n-1}[E].
 \end{gathered}
\end{equation}
These classes are K\"ahler, and the blow-up intersection formulas give
$c_\varepsilon=n\beta_\varepsilon\alpha_\varepsilon^{n-1}/
\alpha_\varepsilon^n=m$. The defects on $Y'$ are unchanged, whereas,
for $W\subset E$ of dimension $p$,
\[
 \mathcal J_{p,\varepsilon}(W)
 =m\left(1-\frac{p}{n-1}\right)\varepsilon^p
   \deg_{\mathcal O_E(1)}W.
\]
For every other strict transform $\widetilde V$, put
$\mu=\operatorname{mult}_qV$. Since $\xi$ embeds a neighborhood of $q$,
$\mu\le\xi^p\cdot V\le\xi A^{p-1}\cdot V$, and therefore
\[
 \mathcal J_{p,\varepsilon}(\widetilde V)
 \ge\left[\kappa_p-pt_\varepsilon
 -m\left(1-\frac{p}{n-1}\right)\varepsilon^p\right]
 \xi A^{p-1}\cdot V>0
\]
for sufficiently small $\varepsilon$, uniformly in $V$.
Thus
\[
 \operatorname{Null}_J(\alpha_\varepsilon,\beta_\varepsilon)
 =Y'\sqcup E\cong Y\sqcup\mathbb P^{n-1},
\]
and these are the only positive-dimensional irreducible null
subvarieties. No smooth boundary-cone representative exists: its
hyperplane-trace inequality would be strict on every $m$-plane,
since $m<n-1$, contradicting $\mathcal J_{m,\varepsilon}(Y')=0$.

We next prove global boundedness for fixed $\varepsilon$. Choose
arbitrary K\"ahler forms $\omega\in\alpha_\varepsilon$ and
$\chi\in\beta_\varepsilon$. The normalized smooth approximants
$u_j\to u$ in $L^1(X)$ satisfy
\[
 \omega_j=(1+s_j)\omega+dd^cu_j>0,\qquad
 G_j(\omega_j)=m,\qquad
 G_j(T):=\operatorname{tr}_T\chi+\lambda_j\det_T\chi,
\]
where $s_j,\lambda_j\ge0$ tend to zero. They are uniformly locally
bounded off $Y'\cup E$ by the preceding estimates.
For $D=Y'$ or $E$, put $p_D=\dim D$, $k_D=n-p_D$. The bundle chart
and the local model of the blow-up identify a neighborhood of $D$
with a neighborhood of the zero section in
\[
 N_{D/X}=(M_D^{-1})^{\oplus k_D},\qquad
 M_{Y'}=L,\qquad M_E=\mathcal O_E(1).
\]
Let $\pi_D$ be the projection. Since
$\eta_D:=(p_D/m)\chi|_D\in[\omega|_D]$, choose smooth $g_D$ with
$\omega|_D+dd^cg_D=\eta_D$.
Equip $M_D$ with positive curvature $\kappa_D$, let $t$ be the squared
norm for the dual direct-sum metric, and set
\[
 w_D=g_D\circ\pi_D+t^{1/4},\qquad T_D=\omega+dd^cw_D.
\]
The cusp is bounded and continuous. In the Chern splitting,
$dd^ct^{1/4}$ has horizontal block $\frac14t^{1/4}\kappa_D$,
zero mixed block, and positive vertical block comparable to
$t^{-3/4}I_{k_D}$. Thus $T_D>0$ on a punctured tube, and block inversion gives
\begin{align}
 \operatorname{tr}_{T_D}\chi
 &=m-\frac{m}{4p_D}t^{1/4}
       \operatorname{tr}_{\eta_D}\kappa_D+O(t^{1/2}),
 \label{eq:J-mixed-cusp-trace}\\
 \det_{T_D}\chi&=O(t^{3k_D/4}).
 \label{eq:J-mixed-cusp-determinant}
\end{align}
Since $\sup_j\lambda_j<\infty$, choose disjoint fixed tubes $U_D$
such that $G_j(T_D)<m$ on $U_D\setminus D$ for all $j$.
The local estimate on $\partial U_D$ gives $u_j\ge w_D-C_D$ there.
An interior minimum of $u_j-w_D$ off $D$ would give
$\omega_j\ge T_D$, contradicting $G_j(\omega_j)=m>G_j(T_D)$.
A minimum on $D$ is also impossible, since the normal growth
$t^{1/4}\asymp|z|^{1/2}$ precludes a smooth upper test for $w_D$.
Hence $u_j\ge w_D-C_D$ throughout $U_D$. Applying this to both
components and using the local estimate on the complement yields
\begin{equation}\label{eq:J-mixed-global-bound}
 \sup_j\|u_j\|_{L^\infty(X)}<\infty,
 \qquad u\in L^\infty(X).
\end{equation}
The bound may depend on the fixed $\varepsilon$ and background forms.
\end{example}

It remains an interesting problem to determine when $u$ is globally bounded; Fu--Zhang~\cite{FuZhang} establish such boundedness under the existence of a smooth boundary-cone subsolution.

\appendix
\section{Mass concentration}
\label{app:full-trace-concentration}
\input{AppendixMassConcentration}
\raggedbottom
\printbibliography




\end{document}

%% file: j-intro-0911.tex

\section{Introduction}\label{sec:introduction}

In this paper, we study the semistable $J$-equation on compact K\"ahler
manifolds, with particular emphasis on the geometry of its numerical null
locus and the regularity of Murakami's weak solution.

\subsection{Smooth solvability and the numerical criterion}

The $J$-equation is a fully nonlinear equation for K\"ahler metrics in a
fixed cohomology class. It arose from Donaldson's moment-map picture
\cite{Donaldson99} and X.~X.~Chen's study of the
$J$-functional and its relation to the Mabuchi energy
\cite{Chen00,Chen04}. Let $X$ be a compact connected K\"ahler manifold of
complex dimension $n\ge2$, and let $\omega,\chi$ be K\"ahler forms. Write
$$
 \alpha=[\omega],\qquad \beta=[\chi],\qquad
 c=n\frac{\beta\cdot\alpha^{n-1}}{\alpha^n}.
$$
We use $d^c=\frac{\sqrt{-1}}2(\bar\partial-\partial)$. For a K\"ahler potential $\phi$ with
$\omega_\phi:=\omega+\ddc\phi>0$,
the $J$-equation is
\begin{equation}\label{eq:J-equation}
 \tr_{\omega_\phi}\chi=c,
 \qquad\text{equivalently}\qquad
 c\,\omega_\phi^n
 = n\,\chi\wedge\omega_\phi^{n-1}.
\end{equation}

The associated $J$-flow is its gradient-flow formulation \cite{Chen04,SongWeinkove}.
The smooth solvability problem admits both an analytic and a numerical
characterization.

Following Weinkove's convergence results \cite{Weinkove04,Weinkove06},
Song--Weinkove established the sharp analytic criterion
\cite{SongWeinkove}: smooth solvability, smooth convergence of
the $J$-flow, and the existence of a K\"ahler representative
$\widehat\omega\in\alpha$ satisfying
\begin{equation}\label{eq:SW-cone-intro}
 c\widehat\omega^{n-1}
 -(n-1)\chi\wedge\widehat\omega^{n-2}>0
\end{equation}
are equivalent. Thus the analytic condition is the existence of a strict
subsolution, expressed by positivity of an $(n-1,n-1)$-form.
Fang--Lai--Ma extended this cone-condition approach to a family of inverse
$\sigma_k$ flows \cite{FangLaiMa}.

Lejmi--Sz\'ekelyhidi conjectured an intersection-theoretic characterization
of the same condition \cite{LejmiSzekelyhidi}. For each irreducible
subvariety $V\subsetneq X$ with $1\le p=\dim V<n$, set
$$
 \mathcal J_p(V)=c\,\alpha^p\cdot[V]
       -p\,\beta\cdot\alpha^{p-1}\cdot[V].
$$
Their conjecture asserts that the strict  condition
\eqref{eq:SW-cone-intro} is equivalent to $\mathcal J_p(V)>0$ for every
such $V$. It thereby relates the existence of a solution to positivity
on subvarieties, in the spirit of a Nakai--Moishezon criterion.
Collins--Sz\'ekelyhidi proved this characterization in the toric setting
\cite{CollinsSzekelyhidi}.

Gao Chen proved the equivalence with uniform $J$-stability and a uniform
strengthening of the intersection inequalities
\cite{Chen21}. Datar--Pingali removed the uniform margin in
the projective setting and treated generalized Monge--Amp\`ere equations
\cite{DatarPingali}. Song established the sharp strict
numerical criterion on arbitrary compact K\"ahler manifolds
\cite{Song20}, settling the conjecture of  Lejmi--Sz\'ekelyhidi.
Fang--Ma subsequently obtained
analytic and numerical solvability criteria in a broader differential-form
framework \cite{FangMa24}.
The strict theory thus identifies smooth solvability with positivity of
all numerical defects. The natural boundary problem is to allow
some of these defects to vanish.

\subsection{The semistable boundary and the null locus}

We assume throughout that $(\alpha,\beta)$ is numerically $J$-semistable:
\begin{equation}\label{eq:J-semistability}
 \mathcal J_p(V)\ge0\quad
 (V\subsetneq X\text{ irreducible},\ 1\le p=\dim V<n).
\end{equation}
Unlike the strict criterion, \eqref{eq:J-semistability} allows equality on proper
subvarieties. The analytic problem is to obtain a
weak solution, and understand  its
geometric degeneration.

An earlier model is the surface boundary result of
Fang--Lai--Song--Weinkove \cite{FangLaiSongWeinkove}.
Assuming that $c\alpha-\beta$ has a smooth semipositive representative,
they obtained a uniform $C^0$ estimate for the $J$-flow and smooth
convergence to a singular K\"ahler metric away from finitely many curves
of negative self-intersection. Related boundary results under divisor-controlled
or smooth boundary-cone hypotheses were obtained by  Sun
\cite[Theorem~1.1]{Sun24} and T\^o  \cite[Theorem~1.1]{To23}.

Murakami's weak-solution theory supplies the analytic starting point of
the present paper. On K\"ahler surfaces he proved current convergence
under numerical $J$-nefness, without assuming a smooth semipositive
representative \cite[Theorem~1.2]{MurakamiSurface}. In arbitrary dimension,
he established boundary existence, admissible uniqueness, and current
convergence for generalized Monge--Amp\`ere equations
\cite[Theorems~1.7 and~1.10]{Murakami}. Applied to the $J$-equation, his
boundary argument identifies limits of the compatible strict twisted
problems, and uniqueness gives a canonical normalized weak solution.
We denote this solution by $u$, with $\sup_Xu=0$, and its normalized
smooth elliptic approximants by $u_j$; see
\cite[proof of Theorem~2.5]{Murakami} and
\cref{def:weak-J,thm:smooth-approximants}. The admissible branch includes
the lower-degree current inequalities as well as the top-degree equation.

From the intersection-theoretic point of view, the numerical equality set is the $J$-null locus
\cite{Fu26,LiuThreefold}:
\begin{equation}\label{nullj}
 \NullJ(\alpha,\beta)=
 \bigcup_{\substack{V\subsetneq X\ \mathrm{irreducible}\\
                 1\le\dim V<n,\ \mathcal J_{\dim V}(V)=0}}V.\end{equation}
The central question we are trying to address  in this paper is whether it describes a single
intrinsic degeneration locus from three perspectives: does \eqref{nullj}
form a proper analytic subset, can it be realized as the unavoidable
singular set of suitable strict currents, and does it detect exactly
where the canonical weak solution loses regularity?

For ordinary positivity, the model is the theorem of Collins--Tosatti
\cite{CollinsTosatti}: the numerical null locus of a nef and big
$(1,1)$-class equals its non-K\"ahler locus, defined through the singularities
of K\"ahler currents. The $J$-equation asks for a nonlinear counterpart of
this picture. In addition to the positivity of a current, the analytic
object must also retain the strict cone condition underlying
\eqref{eq:SW-cone-intro}. We construct such an object and relate it to the
regularity of Murakami's solution.




\begin{theorem}\label{thm:main}
Let $u$ be the unique normalized weak solution for semistable $J$-equation. Then  $$u\in C^{\infty}(X\setminus \NullJ(\alpha,\beta)).$$
Moreover, $\NullJ(\alpha,\beta)$ is a proper analytic subset with finitely many irreducible components. 
\end{theorem}
Theorem \ref{thm:main} gives the regularity of the solution outside $\NullJ(\alpha,\beta)$. In particular, it implies that $u$ is bounded away from $\NullJ(\alpha,\beta)$. \cref{rem:nonReg2} shows that $\NullJ$ is the set where $u$ is not locally $C^2$ and hence \cref{thm:main} characterizes $X\setminus \NullJ$ as the largest locus of the $C^2$-regularity of $u$. However, we do not know the unbounded locus of $u$. 
The surface boundary results in
\cite{FangLaiSongWeinkove} and \cite{Fang-Lai} already illustrate why these notions must be
separated. 

\subsection{Relation to recent work}
The semistable $J$-equation and its degeneration loci have recently been 
studied in several settings.
X. Fu proved analyticity of the $J$-null locus on surfaces and toric
manifolds, and smoothness of the torus-invariant weak solution on the
dense complex torus \cite{Fu26}.
Liu proved an analytic characterization of the null locus and smooth
$J$-flow convergence on its complement for K\"ahler threefolds
\cite{LiuThreefold}; a second paper treats higher dimensions under
$J$-bigness and a smooth boundary-cone condition \cite{LiuBoundary}.
Fu--Zhang obtained divisorial rigidity under semistability and global
boundedness under a smooth boundary-cone condition \cite{FuZhang}.
The geometry of obstructing subvarieties is also studied in the
minimal-slope program of Datar--Mete--Song \cite{DatarMeteSong}, the
wall--chamber theory of Khalid--Sj\"ostr\"om Dyrefelt
\cite{KhalidSjoestroemDyrefelt}, and the finiteness and rigidity results
of Sivaram--Sj\"ostr\"om Dyrefelt \cite{SSD}.

During the final preparation of this manuscript, a third paper of Liu
appeared on arXiv \cite{LiuFiniteness}.  Under numerical
$J$-semistability in arbitrary dimension, Liu proves that there are only
finitely many positive-dimensional irreducible $J$-null subvarieties,
together with a uniform positive gap for non-null subvarieties.  In
particular, this is a stronger numerical finiteness statement than the
finiteness of the irreducible components of their union obtained here.
His argument and ours both make essential use of Chen's
mass-concentration construction, but in complementary directions.  Liu
applies concentration to null subvarieties and uses the resulting
exceptional sets in a dimension-descending finiteness argument, whereas
our   argument starts from positive numerical defect, combines it with
Demailly--Kiselman attenuation to construct barriers with singularities in a minimal analytic sets, and excises the corresponding non-null components.

Our theorem treats the full null locus in arbitrary dimension under
numerical $J$-semistability alone, without additional assumptions.
Its main distinction from the results above is a three-way interpretation
of this locus: numerically, it is the union of the $J$-null subvarieties;
analytically, it is the intrinsic minimal strict-barrier locus; and from
the PDE viewpoint, it is exactly the ambient $C^2$-singular locus of
Murakami's canonical weak solution.  In particular, the weak solution is
smooth and the elliptic approximants converge in
$C^\infty_{\mathrm{loc}}$ on its entire complement, while no point of the
null locus admits an ambient $C^2$ neighborhood.  The existence and
uniqueness of the weak solution itself are supplied by Murakami's theory
\cite{Murakami}.

\subsection{Ideas of the proof}
The idea is straightforward: we search for a strict subsolution $v$ that has an analytic pole set and then use it as a barrier to obtain \emph{a prior} estimates. In practice, it consists of 3 steps that have different geometric and analytic flavors. 

The first step combines Chen's mass concentration with Demailly's global
regularization and Kiselman's attenuation of singularities. Chen supplies
a positive current with a strict cone condition
\cite[Theorem~1.18]{Chen21}, used in the local convolution form. We apply Demailly's exponential regularization and
attenuation \cite{DemaillyFlow94} to produce a positive current whose potential has an analytic pole set which still satisfies the strict cone condition. Using finite max operation, we obtain a current satisfies the strict cone condition whose potential has the minimal analytic pole set $Z_J$.  

The second step removes unwanted components of the analytic locus $Z_J$. A component $V$ with positive $\mathcal J_p(V)$
must admit a local potential which satisfies the strict cone condition. With careful estimates, we can glue the potential back to resolve the singularity which contradicts the minimality. 

The third step applies the singular relative $L^\infty$ estimate of
Fang--Ma--Wu \cite{FMWRelative}. A $J$-specific determinant bound gives 
$u_j\ge v-C$ uniformly for each fixed barrier $v$ satisfies the strict cone condition.
A weighted logarithmic-trace estimate, adapted from
\cite{Weinkove06,SongWeinkove}, then yields local metric bounds; standard
elliptic regularity gives smooth convergence off $Z_J$. This proves the
regularity statement in \cref{thm:main}. 

The Chen–Demailly–Kiselman barrier and attenuation framework suggests analogous questions for supercritical deformed Hermitian–Yang–Mills and more general fully nonlinear Kähler equations, but these extensions require additional equation-specific arguments and are left to subsequent work.

The constructions and ideas developed here may also be applied to
deformed Hermitian--Yang--Mills (Leung--Yau--Zaslow) equations and
more general fully nonlinear geometric PDEs. We leave these
directions for future work.

\subsection{Organization of the paper}

\Cref{sec:setup} fixes the conventions and records the weak-solution and
subsolution formulations.  \Cref{sec:global-inputs} constructs strict
subsolutions, regularizes them, and establishes the quantitative growth
estimates used later.  \Cref{sec:minimal-barrier} defines the minimal
$J$-barrier with pole set $J_Z$ and controlled logrithmetic growth. \Cref{sec:component-removal} identifies this
locus with the numerical null locus.  \Cref{sec:determinantal-transfer}
proves the relative $L^\infty$ and weighted metric estimates, obtains
smooth convergence away from the null locus, and completes the proof of
\cref{thm:main}.  \Cref{sec:comparison-examples} compares the resulting
locus with the shifted non-K\"ahler locus and treats an example with null components of different dimensions.
The mass-concentration statement used in the null-locus identification is
recorded in Appendix~\ref{app:full-trace-concentration}.

\medskip

\paragraph{{\textbf{Acknowledgments.}}}
The authors thank Jinyang Wu and Xin Fu for helpful discussions
on related topics.

\medskip

\paragraph{\textbf{Declaration of AI assistance.}}
The main ideas and mathematical arguments are due to the authors.
AI tools were used to assist with the audit of the arguments and
the editing of the manuscript. The authors assume full responsibility
for the mathematical content and the final text.

%% file: IdentificationofNulllocus.tex
It is straightforward to show that the minimal barrier locus  $Z_J$ contains the numerical null locus $\NullJ$. In this section, we will prove that these 2 sets are equal.     

\begin{theorem}
\label{thm:locus-identification}
For $J$-semistable pair $(\alpha,\beta)$, 
$Z_J=\NullJ(\alpha,\beta)$. In particular, $\NullJ$ is a
proper analytic subset of $X$ with finitely many irreducible components.
\end{theorem}

To prove the above theorem, we use \cref{thm:minimal-barrier} to pick a particular barrier $$(\underline v_*,\tau_*,\eta_*,Z_J)\in \mathfrak B_J$$ that realizes the minimal barrier locus $Z_J$. 
\begin{proposition}
\label{prop:uniform-positive-patch}\label{cor:null-in-every-barrier}
Every null subvariety is contained in $Z_J$, and $Z_J$ has no
zero-dimensional irreducible component.
\end{proposition}

\begin{proof}
Let $\underline v_*$ be the realizing barrier for $Z_J$, with
$T_*=(1-\tau_*)\omega+\ddc\underline v_*$ and
$P_\chi(T_*)\le c-\eta_*$. Let
$\omega_j=(1+s_j)\omega+\ddc u_j$ be the approximants of
\cref{thm:smooth-approximants}, with $s_j\downarrow0$.

Suppose that a null subvariety $Y^m$ is not contained in $Z_J$.
Choose $U\Subset X\setminus Z_J$ meeting $Y_{\mathrm{reg}}$ and constants
$C_j$ such that $\underline v_*+C_j>u_j+1$ on $\overline U$.
Near $Z_J$ the smooth branch $u_j$ dominates; choosing the smoothing scale
below these two gaps, the regularized maximum defines
$\Omega_j\in(1+s_j)\alpha$. By \cref{lem:cone-maxima},
$P_\chi(\Omega_j)<c$ on $X$, while on $U$,
\[
 \Omega_j=T_*+(\tau_*+s_j)\omega,
 \qquad P_\chi(\Omega_j)\le c-\eta_*.
\]
Thus $c\Omega_j^m-m\chi\wedge\Omega_j^{m-1}$ is nonnegative on
$Y_{\mathrm{reg}}$ and at least $\eta_*c^{-m}\chi^m$ on $Y\cap U$.
Integration gives
\begin{equation}\label{eq:Z_JcontainNullJ}
 (1+s_j)^{m-1}
 \left(\mathcal J_m(Y)+cs_j\int_Y\omega^m\right)
 \ge \eta_*c^{-m}\int_{Y\cap U}\chi^m>0.
\end{equation}
This contradicts $\mathcal J_m(Y)=0$ as $j\to\infty$, and hence
$\NullJ(\alpha,\beta)\subset Z_J$.

If $x_0$ were an isolated component of $Z_J$, choose a coordinate ball
$B$ about $x_0$ whose boundary annulus misses $Z_J$, and let $h$ be a
local potential of $(1-\tau_*)\omega$.  For $K$ large,
$q=K|z|^2-h-C$ satisfies the strict $P_\chi$-cone condition on $B$.
Choose $C$ so that $q<\underline v_*$ on the boundary annulus.  Replacing
$\underline v_*$ by $\max\{\underline v_*,q\}$ on $B$ gives, by
\cref{lem:cone-maxima}, a barrier that is finite at $x_0$.  This
contradicts the definition of $Z_J$.
\end{proof}

\begin{remark} \label{rem:nonReg2}
The same argument excludes local ambient $C^2$ regularity of Murakami's
weak solution at a point of the null locus.  Indeed, suppose that $u$ is
$C^2$ near $x\in\NullJ(\alpha,\beta)$, and choose an irreducible null
subvariety $Y^m$ containing $x$.  On a relatively compact neighborhood
$U$ meeting $Y_{\mathrm{reg}}$, the equation and compactness give
$P_\chi(\omega_u)\le c-\eta$ for some $\eta>0$.  Choose $C_j$ so that
 $u+C_j>u_j+1$ on $\overline U$ and set
\[
 \Omega_j=(1+s_j)\omega+\ddc\max\{u_j,u+C_j\}.
\]
The weak cone condition for $u$ and \cref{lem:cone-maxima} give
$P_\chi(\Omega_j)\le c$ globally, while the inequality is uniformly
strict on $U$.  The local potential of $\Omega_j$ maximum is bounded. Integration of $c\Omega_j^m-m\Omega_j^{m-1}\wedge \chi$ on $Y_{\mathrm{reg}}$ therefore gives the analogue of
\eqref{eq:Z_JcontainNullJ}, again contradicting
$\mathcal J_{m}(Y)=0$ as $j\to\infty$.
\end{remark}

The reverse inclusion will be proved by excising every irreducible
component $Y\subset Z_J$ for which $\mathcal J_{\dim Y}(Y)>0$.  We will find in a neighborhood of $Y$ a local strict barrier with logarithmic singularities along a proper analytic subset of $Y$. Then we glue with a barrier with minimal barrier locus on $X$ to excise generic points of $Y$. This will result a contradiction since $Z_J$ is minimal.

Fix an irreducible subvariety $Y\subsetneq X$ of dimension
$1\le m\le n-1$.  Choose an embedded
resolution
\[
 \mu:\widetilde X\longrightarrow X
\]
of the pair $(Y,Z_0)$ which is an isomorphism over $X\setminus Z_0$.
Let $\widetilde Y$ be the smooth strict transform of $Y$, put
$f=\mu|_{\widetilde Y}$, and fix a K\"ahler form $\varpi$ on
$\widetilde X$.  We write $\varpi_Y=\varpi|_{\widetilde Y}$.

\begin{proposition}[Relative local cone patch]\label{lem:local-cone-patch}
Notations as above. Suppose that $\mathcal J_m(Y)>0$, and let
$Z_0\subsetneq Y$ be a proper analytic subset containing
$Y_{\mathrm{sing}}$.  There exists a proper analytic set $Z$ in $Y$ containing $Z_0$, 
constants $0<\tau<1$, $\eta>0$, $\ell>0$, an integer $k\ge1$, and
$\varepsilon>0$ with the following property.  Set
\[
 d_Z(x)=\min\{1,\dist(x,Z)\},\qquad
 d_Y(x)=\min\{1,\dist(x,Y)\},
\]
and define
\[
 \mathcal U=\{x\in X\setminus Z:
        d_Y(x)<\varepsilon d_Z(x)^k\},
\]
then there is a smooth function $q$ on $\mathcal U$ such that
\begin{equation}\label{eq:local-ambient-P-patch}
 P_\chi\bigl((1-\tau)\omega+\ddc q\bigr)\le c-\eta
 \quad\text{on }\mathcal U,
\end{equation}
and
\begin{equation}\label{eq:local-patch-log-bound}
 q(x)\le \ell\log d_Z(x)+C.
\end{equation}
\end{proposition}

We divide the proof into 4 steps. 
\medskip
\noindent\textit{Step 1}. We first use mass concentration to find a positive current on $\widetilde Y$.
Let $\omega_j=(1+s_j)\omega+\ddc u_j$ be the forms from
\cref{thm:smooth-approximants}, and set
$\eta_j=c-\max_XP_\chi(\omega_j)>0$.  Choose
$\varepsilon_j\downarrow0$ and
$0<\delta_j\le\varepsilon_j\eta_j/(2m)$, and put
\[
 \widetilde\omega_j=f^*\omega_j+\varepsilon_j\varpi_Y,
 \qquad
 \widetilde\chi_j=f^*\chi+\delta_j\varpi_Y.
\]
The rank inequality gives
\[
 Q_{f^*\chi}(\widetilde\omega_j)\le P_\chi(\omega_j),
 \qquad
 Q_{\widetilde\chi_j}(\widetilde\omega_j)<c,
\]
and the projection formula gives
\[
 c[\widetilde\omega_j]^m
 -m[\widetilde\chi_j]\cdot[\widetilde\omega_j]^{m-1}
 \longrightarrow\mathcal J_m(Y)>0.
\]
Thus \cref{prop:full-trace-concentration} applies directly after multiplying a small $1-\tau$ $\tau>0$, we obtain
a K\"ahler current $S\in (1-\tau)f^*\alpha$ and constants
$a,\lambda,\eta_1>0$ such that
\begin{equation}\label{eq:strict-current-on-resolution}
 S\ge a\varpi_Y,
 \qquad
 Q_{f^*\chi+\lambda\varpi_Y}(S)\le c-4\eta_1
\end{equation}
in the local-convolution sense.

\medskip
\noindent\textit{Step 2.} Next, we  extend $S$ to a smooth current $(1-\tau)\mu^*\omega+\ddc \widetilde q$ on a resolved tube $\widetilde{\mathcal U}$ such that  
\begin{equation}\label{eq:resolved-ambient-P-patch}
 P_{\mu^*\chi+\lambda\varpi}
 \bigl((1-\tau)\mu^*\omega+\ddc\widetilde q\bigr)
 \le c-\eta_1, \qquad \text{on }\ \widetilde{\mathcal U}
\end{equation}

Apply
\cref{thm:global-cone-regularization,prop:global-growth} on
$\widetilde Y$ with $F=Q$ to $S$ in Step 1.  As in the proof of
\cref{thm:strict-barrier}, insert a sufficiently small logarithmic pole
along $f^{-1}(Z_0)$ and the analytic set produced by regularization.
This gives a proper analytic set $\widetilde Z\subsetneq\widetilde Y$
containing $f^{-1}(Z_0)$ and a function
$\varphi\in C^\infty(\widetilde Y\setminus\widetilde Z)$ such that, for
\[
 \Theta=(1-\tau)\mu^*\omega|_{\widetilde Y}+\ddc\varphi,
 \qquad
 \widehat\chi=(\mu^*\chi+\lambda\varpi)|_{\widetilde Y},
\]
one has
\begin{equation}\label{eq:resolved-tangential-Q}
 Q_{\widehat\chi}(\Theta)\le c-3\eta_1.
\end{equation}
Define
\(
 d_{\widetilde Z}(\widetilde y)
 =\min\{1,\dist_{\widetilde Y}(\widetilde y,\widetilde Z)\},
\) and \(
 d_{\widetilde Y}(\widetilde x)
 =\min\{1,\dist_{\widetilde X}(\widetilde x,\widetilde Y)\}.
\)
Then
\begin{equation}\label{eq:resolved-tangential-growth}
 \varphi(\widetilde y)
 \le\ell\log d_{\widetilde Z}(\widetilde y)+C,
 \qquad
 |\nabla\varphi(\widetilde y)|
 \le C d_{\widetilde Z}(\widetilde y)^{-N}.
\end{equation}

Choose finitely many adapted coordinate neighborhoods
$B_i'\Subset B_i\subset\widetilde X$ such that
$B_i'\cap\widetilde Y$ cover $\widetilde Y$ and
$\widetilde Y\cap B_i=\{w_i=0\}$ in coordinates $(z_i,w_i)$.  If
$h_i$ is a local potential of $(1-\tau)\mu^*\omega$, put
\[
 q_i(z_i,w_i)
 =\varphi(z_i)+h_i(z_i,0)+K|w_i|^2-h_i(z_i,w_i).
\]
Then
\[
 (1-\tau)\mu^*\omega+\ddc q_i
 =
 \begin{pmatrix}
  \Theta(z_i)&0\\
  0&K\,\ddc|w_i|^2
 \end{pmatrix}.
\]
Since $\widehat\chi\ge\lambda\varpi_Y$,
\eqref{eq:resolved-tangential-Q} gives a uniform positive lower bound
for $\Theta$.  We may choose $K$ independently of $i$ and shrink the
normal directions so that
\begin{equation}\label{eq:local-extension-branches}
 (1-\tau)\mu^*\omega+\ddc q_i\ge a_0\varpi,
 \qquad
 P_{\mu^*\chi+\lambda\varpi}
 \bigl((1-\tau)\mu^*\omega+\ddc q_i\bigr)
 \le c-2\eta_1.
\end{equation}

Choose $\vartheta_i:B_i\to[-1,0]$ which vanishes near $B_i'$ and is
equal to $-1$ near the outer boundary of $B_i$.  For a sufficiently
small fixed $d>0$, the functions
$\widehat q_i=q_i+d\vartheta_i$ satisfy
\begin{equation}\label{eq:penalized-extension-branches}
 (1-\tau)\mu^*\omega+\ddc\widehat q_i
 \ge\frac{a_0}{2}\varpi,
 \qquad
 P_{\mu^*\chi+\lambda\varpi}
 \bigl((1-\tau)\mu^*\omega+\ddc\widehat q_i\bigr)
 \le c-\frac32\eta_1.
\end{equation}
On overlaps, \eqref{eq:resolved-tangential-growth} gives
\begin{equation}\label{eq:extension-overlap}
 |q_i(\widetilde x)-q_{i'}(\widetilde x)|
 \le C d_{\widetilde Y}(\widetilde x)
 d_{\widetilde Z}(\pi(\widetilde x))^{-N}.
\end{equation}

In a tubular neighborhood $U$ of $\widetilde Y$ in $\widetilde X$, let $\pi: U\to \widetilde Y$ be the nearest point projection. Fix $k_0>N+1$ and take $\varepsilon_0>0$ sufficiently small. \begin{equation}\label{eq:controlled-resolved-tube}
 \widetilde{\mathcal U}
 =\left\{\widetilde x:\pi(\widetilde x)\in
 \widetilde Y\setminus\widetilde Z,\quad
 d_{\widetilde Y}(\widetilde x)
 <\varepsilon_0
 d_{\widetilde Z}(\pi(\widetilde x))^{k_0}\right\}.
\end{equation}  On
$\widetilde{\mathcal U}$, \eqref{eq:extension-overlap} gives
$|q_i-q_{i'}|<d/4$ whenever both branches are defined.  If
$\widetilde x$ approaches the outer boundary of $B_i$, some
$B_{i'}'$ still contains $\pi(\widetilde x)$ and
\[
 \widehat q_i(\widetilde x)
 \le\widehat q_{i'}(\widetilde x)-\frac{3d}{4}.
\]
Hence the regularized maximum of the active branches
\[
 \widetilde q
 =\widetilde{\max}_{\delta}(\widehat q_i),
 \qquad 0<\delta<\frac d8,
\]
defines a smooth function on $\widetilde{\mathcal U}$; see
\cite[I.5.18]{Demailly}.  By \cref{lem:cone-maxima}, it satisfies
\eqref{eq:resolved-ambient-P-patch}.  The upper bound for the
regularized maximum, together with
\eqref{eq:resolved-tangential-growth} and
\eqref{eq:extension-overlap}, gives
\begin{equation}
    \quad \widetilde q(\widetilde x)
 \le\ell\log d_{\widetilde Z}(\pi(\widetilde x))+C, \qquad \text{on }\ \widetilde{\mathcal U}, \label{eq:resolved-patch-log-bound}
\end{equation}
for some $\ell,C>0$.

\medskip
\noindent\textit{Step 3.} Next, we would like to push $\widetilde q$ to $\mathcal{U}$. We need to compare with $\mu(\widetilde{\mathcal {U}})$ and $\mathcal{U}$.

\begin{lemma}
\label{lem:resolved-tube-comparison}
Notations as above. Let $Z=\mu(\widetilde Z)$. There are constants
$A,C>0$ and an integer $L\ge1$ with the following properties.  If
$x\in X\setminus Z$, $d_Y(x)<\frac12d_Z(x)$, and
$\widetilde x=\mu^{-1}(x)$, then
\begin{equation}\label{eq:resolved-normal-distance}
 d_{\widetilde Y}(\widetilde x)
 \le C d_Z(x)^{-A}d_Y(x).
\end{equation}
For $\widetilde y\in\widetilde Y$ one also has
\begin{equation}\label{eq:resolved-set-comparison}
 d_{\widetilde Z}(\widetilde y)^L
 \le C d_Z(f(\widetilde y)).
\end{equation}
Moreover, if $k_0>L$ and 
$k>A+k_0$, there is $\varepsilon>0$ such that
\begin{equation}\label{eq:downstairs-inside-upstairs-tube}
 \mathcal U=\{x\in X\setminus Z:
 d_Y(x)<\varepsilon d_Z(x)^k\}
 \subset
 \mu\bigl(\widetilde{\mathcal U}\setminus
 (\widetilde Z\cup\operatorname{Exc}(\mu))\bigr).
\end{equation}
For $x\in\mathcal U$, with
$\widetilde y=\pi(\mu^{-1}(x))$, one has
\begin{equation}\label{eq:two-sided-resolved-distance}
 C^{-1}d_Z(x)
 \le d_{\widetilde Z}(\widetilde y)
 \le C d_Z(x)^{1/L}.
\end{equation}
\end{lemma}

\begin{proof}
Since $f^{-1}(Z_0)\subset\widetilde Z$ and $\mu$ is injective away
from $Z_0$, one has $f^{-1}(Z)=\widetilde Z$.  Write
$\nu=\mu^{-1}$ on $X\setminus Z$.  If
$E=\operatorname{Crit}(\mu)$, then
$E\subset\mu^{-1}(Z_0)\subset\mu^{-1}(Z)$.  The \L ojasiewicz
inequality and the Lipschitz bound for $\mu$ give
\[
 |\operatorname{Jac}\mu(\widetilde z)|
 \ge C^{-1}\dist(\widetilde z,E)^A
 \ge C^{-1}d_Z(\mu(\widetilde z))^A.
\]
The cofactor formula for $(D\mu)^{-1}$ therefore gives
\begin{equation}\label{eq:inverse-modification-derivative}
 \lVert D\nu(z)\rVert\le C d_Z(z)^{-A}.
\end{equation}
Choose $x_Y\in Y$ with $\dist_X(x,x_Y)=d_Y(x)$.  If
$d_Y(x)<\frac12d_Z(x)$, a minimizing geodesic from $x$ to $x_Y$ stays
at distance at least $\frac12d_Z(x)$ from $Z$.  Since
$\nu(x_Y)\in\widetilde Y$, integrating
\eqref{eq:inverse-modification-derivative} along this geodesic proves
\eqref{eq:resolved-normal-distance}.

When $\widetilde x$ belongs to the domain of $\pi$, set
$\widetilde y=\pi(\widetilde x)$.  The Lipschitz bound for $\mu$ and
the triangle inequality give
\begin{equation}
    \label{eq:resolved-singular-distance1}
 d_Z(x)
 \le C\bigl(d_{\widetilde Y}(\widetilde x)
 +d_{\widetilde Z}(\widetilde y)\bigr).
\end{equation}

If
$g_1,\ldots,g_s$ are local generators of the ideal of $Z$, the common
zero set of $g_\alpha\circ f$ is $\widetilde Z$.  The \L ojasiewicz
inequality on $\widetilde Y$, followed by
$|g_\alpha|\le C d_Z$, gives \eqref{eq:resolved-set-comparison}.

It remains to prove \eqref{eq:downstairs-inside-upstairs-tube}.  Let
$x\in\mathcal U$ and $\widetilde x=\mu^{-1}(x)$.  Taking
$\varepsilon<\frac12$, \eqref{eq:resolved-normal-distance} gives
\begin{equation}
    \label{eq:resolved-singular-distance}
 d_{\widetilde Y}(\widetilde x)
 \le C\varepsilon d_Z(x)^{k-A}.
\end{equation}
Since $k-A>k_0\ge1$ and $d_Z(x)\le1$, decreasing $\varepsilon$
places $\widetilde x$ in the domain of $\pi$.  Put
$\widetilde y=\pi(\widetilde x)$.  After decreasing $\varepsilon$ once
more, \eqref{eq:resolved-singular-distance1} gives
\[
 d_Z(x)
 \le C\bigl(d_{\widetilde Y}(\widetilde x)
 +d_{\widetilde Z}(\widetilde y)\bigr)
 \le\frac12d_Z(x)+C d_{\widetilde Z}(\widetilde y).
\]
Thus
\(
 d_{\widetilde Z}(\widetilde y)\ge C^{-1}d_Z(x),
\)
and hence
\[
 d_{\widetilde Y}(\widetilde x)
 \le C\varepsilon d_Z(x)^{k-A}
 \le C\varepsilon d_Z(x)^{k_0}
 \le C'\varepsilon
 d_{\widetilde Z}(\widetilde y)^{k_0}.
\]
Taking $C'\varepsilon<\varepsilon_0$ proves that
$\widetilde x\in\widetilde{\mathcal U}$.  Since $x\notin Z$ and the
exceptional image is contained in $Z_0\subset Z$, this proves
\eqref{eq:downstairs-inside-upstairs-tube} and the lower bound in
\eqref{eq:two-sided-resolved-distance}.

For the upper bound, put $y=f(\widetilde y)$.  From
\eqref{eq:resolved-set-comparison}, the Lipschitz bound for $\mu$, and
the definition of $\widetilde{\mathcal U}$,
\[
 d_{\widetilde Z}(\widetilde y)^L\le C d_Z(y)
 \le C\bigl(d_Z(x)+\dist_X(x,y)\bigr)
 \le C d_Z(x)+C\varepsilon_0
 d_{\widetilde Z}(\widetilde y)^{k_0}.
\]
Since $k_0>L$ and $d_{\widetilde Z}(\widetilde y)\le1$, the last term
can be absorbed after decreasing $\varepsilon_0$.  Therefore
\(
 d_{\widetilde Z}(\widetilde y)^L\le C d_Z(x),
\)
which proves the upper bound in
\eqref{eq:two-sided-resolved-distance}.
\end{proof}

\medskip
\begin{proof}[Step 4.  Proof of \cref{lem:local-cone-patch}]
Use the current $S$ obtained in Step~1.  In step 2, we obtaining
$\widetilde Z$, $\tau$, $\eta_1$, $N$, and $\ell$, and $\widetilde q$.  Set
$Z=\mu(\widetilde Z)$.  Since
$\mu(\operatorname{Exc}(\mu))\subset Z_0$ and
$f^{-1}(Z_0)\subset\widetilde Z$, Remmert's theorem gives
\(
 Z_0\subset Z\subsetneq Y.
\)

Let $A$ and $L$ be the exponents in
\cref{lem:resolved-tube-comparison}.  Choose
$k_0>\max\{N+1,L\}$ and then $\varepsilon_0>0$ sufficiently small.
The extension lemma gives $\widetilde q$ on
$\widetilde{\mathcal U}$.
Choose any integer $k>A+k_0$ and then
$\varepsilon>0$ as in \cref{lem:resolved-tube-comparison}.  By
\eqref{eq:downstairs-inside-upstairs-tube}, $\widetilde q$ descends to
a smooth function $q$ on
\[
 \mathcal U=\{x\in X\setminus Z:
 d_Y(x)<\varepsilon d_Z(x)^k\}.
\]
Dropping the positive term $\lambda\varpi$ from the reference form in
\eqref{eq:resolved-ambient-P-patch} gives
\eqref{eq:local-ambient-P-patch}.  Moreover,
\eqref{eq:resolved-patch-log-bound} and
\eqref{eq:two-sided-resolved-distance} give
\[
 q(x)
 \le\ell\log d_{\widetilde Z}(\pi(\mu^{-1}(x)))+C
 \le\frac{\ell}{L}\log d_Z(x)+C.
\]
This is \eqref{eq:local-patch-log-bound}.
\end{proof}

Now we finish the proof of $\NullJ=Z_J$.
\begin{proof}[Proof of \cref{thm:locus-identification}]
By \cref{prop:uniform-positive-patch},
$\NullJ(\alpha,\beta)\subset Z_J$, and $Z_J$ has no zero-dimensional
component.  

Suppose that an irreducible component $Y$ of $Z_J$, of
dimension $m\ge1$, satisfies $\mathcal J_m(Y)>0$.  Write
\[
 Z_J=Y\cup Y',
\]
where $Y'$ is the union of the other irreducible components, and set
\(
 Z_0=Y_{\mathrm{sing}}\cup(Y\cap Y').
\)
Apply \cref{lem:local-cone-patch}, obtaining $Z$, $k$, $\varepsilon$,
and $q$.  Since $Y\cap Y'\subset Z$, the analytic separation inequality
allows us to increase $k$ and decrease $\varepsilon$ so that
\[
 \dist(x,Y')>d_Y(x),\qquad x\in\mathcal U.
\]
Thus $\dist(x,Z_J)=d_Y(x)$ on $\mathcal U$.  If $Y'=\varnothing$,
this separation is unnecessary.

Let $\underline v_*$ be the realizing barrier, with
$T_*=(1-\tau_*)\omega+\ddc\underline v_*$ and
$P_\chi(T_*)\le c-\eta_*$.  For every sufficiently large $j$, put
\[
 \lambda_j=\frac{2s_j}{\tau_*+2s_j},\qquad
 v_j=\lambda_j\underline v_*+(1-\lambda_j)u_j,
 \qquad
 \tau_j=\frac{\lambda_j\tau_*}{2}.
\]
Then
\[
 (1-\tau_j)\omega+\ddc v_j
 =\lambda_j T_*+(1-\lambda_j)
   \bigl((1+s_j)\omega+\ddc u_j\bigr),
\]
and hence
\begin{equation}\label{eq:lower-slope-cone}
 P_\chi\bigl((1-\tau_j)\omega+\ddc v_j\bigr)
 \le c-\lambda_j\eta_*.
\end{equation}
The logarithmic lower bound for the realizing barrier gives
\begin{equation}\label{eq:lower-slope-growth}
 v_j(x)\ge
 \lambda_j L_*\log\dist(x,Z_J)-C_j
\end{equation}
for some $L_*>0$.

On the collar
\(
\mathcal{C}= \{\frac{\varepsilon}{2}d_Z(x)^k
 <d_Y(x)<\varepsilon d_Z(x)^k\},
\)
\eqref{eq:local-patch-log-bound} and
\eqref{eq:lower-slope-growth} give
\[
 q(x)\le\ell\log d_Z(x)+C,
 \quad
 v_j(x)\ge\lambda_jL_*k\log d_Z(x)-C_j 
\]
Choose $j$ sufficiently large that
$\tau_j\le\tau$ and $\lambda_jL_*k<\ell$.  Then $q-v_j$ is bounded
above on  $\mathcal{C}$.  After subtracting a constant $C_*$ from $q$, set
\[
 v=
 \begin{cases}
  \max\{v_j,q-C_*\},&\text{on }\mathcal U,\\
  v_j,&\text{on }X\setminus\mathcal U.
 \end{cases}
\]
The two definitions agree near the outer boundary of $\mathcal U$.
Since $\tau_j\le\tau$, the local branch remains strict after adding
$(\tau-\tau_j)\omega$.  Thus
\cref{lem:cone-maxima}, together with
\eqref{eq:local-ambient-P-patch} and
\eqref{eq:lower-slope-cone}, shows that
$v$ is a strict barrier after normalization.

Its pole set is $Z\cup Y'$.  Indeed, both branches tend to $-\infty$
toward $Z$, the glued potential equals $v_j$ near
$Y'\setminus Z$, and $q$ is finite on $Y\setminus Z$.  Since
$Z\subsetneq Y$, this contradicts the minimality of $Z_J$.  Therefore
every irreducible component $Y$ of $Z_J$ satisfies
$\mathcal J_{\dim Y}(Y)=0$.  Hence
$Z_J\subset\NullJ(\alpha,\beta)$, while the reverse inclusion follows
from \cref{prop:uniform-positive-patch}.  Since \(Z_J\) is a closed analytic subset of the compact manifold \(X\), it has only finitely many irreducible components.
\end{proof}

%% file: AppendixMassConcentration.tex
We record the two forms of the Demailly--P\u aun concentration argument
used in the paper. The diagonal construction is taken from
\cite[Lemma~2.1 and Proposition~2.6]{DemaillyPaun}; the twisted equation
and the extraction of a cone current follow
\cite[Theorems~1.14 and~1.18]{Chen21}.

\begin{proposition}[Mass concentration for both $P$ and $Q$]
\label{prop:full-trace-concentration}
Let $(M^n,\gamma)$ be a connected compact K\"ahler manifold, let $\chi\ge0$ be a
smooth closed $(1,1)$-form, and put
$\chi_j=\chi+\delta_j\gamma$, where $\delta_j\downarrow0$.
Let $F\in\{P,Q\}$. Suppose that there are K\"ahler forms
$\omega_j\in\alpha_j$ such that
\[
 \alpha_j\longrightarrow\alpha,\qquad
 F_{\chi_j}(\omega_j)<c.
\]
Set
\(
 D_j:=c\,\alpha_j^n-n[\chi_j]\cdot\alpha_j^{n-1}\ge0,\) \(D_j\longrightarrow D.
\)
Then:
\begin{enumerate}[label=\textup{(\roman*)}]
\item if $D>0$, there are $a,\lambda,\eta>0$ and a K\"ahler current
$T\in\alpha$ such that
\[
 T\ge a\gamma,\qquad
 F_{\chi+\lambda\gamma}(T)\le c-\eta
\]
in the local-convolution sense;
\item the same conclusion holds if $D=0$, $F=P$, $\alpha$ is a
K\"ahler class, and $\chi$ is a K\"ahler form.
\end{enumerate}
\end{proposition}

\begin{proof}
After rescaling $\gamma$, we may assume that
$\int_M\gamma^n=1$.

Suppose first that $D=0$ and $F=P$. For every irreducible subvariety
$V\subset M$ of dimension $1\le p<n$, restriction to
$V_{\mathrm{reg}}$ and integration give
\[
 c\,\alpha_j^p\cdot[V]
 -p[\chi_j]\cdot\alpha_j^{p-1}\cdot[V]>0.
\]
Passing to the limit, together with $D=0$ in top degree, shows that
$(M,\alpha,[\chi])$ is $J$-semistable. The perturbed solvability
theorem of \cite{Chen21,Song20}, followed by
\cite[Theorem~1.18]{Chen21}, gives $\varepsilon>0$ and a closed positive
current
\[
 S\in\alpha-\varepsilon[\chi],\qquad
 P_\chi(S)\le c
\]
in the local-convolution sense. Put
$T=S+\varepsilon\chi\in\alpha$. Applying
\cref{lem:elementary} to the absolutely continuous parts and then using
\cref{lem:ac-cone}, we obtain
\[
 P_\chi(T)
 \le \frac{c}{1+\varepsilon c/n}=:c_0<c.
\]
Since $\chi$ is K\"ahler, $T\ge\varepsilon\chi\ge a\gamma$ for some
$a>0$, and hence $Q_\gamma(T)\le n/a$. Therefore
\[
 P_{\chi+\lambda\gamma}(T)
 \le P_\chi(T)+\lambda Q_\gamma(T)
 \le c_0+\lambda n/a.
\]
Choosing $\lambda>0$ so that
$\lambda n/a\le(c-c_0)/2$ proves \textup{(ii)} with
$\eta=(c-c_0)/2$.

Now assume that $D>0$. If $F=P$, then $D_j>0$ for all sufficiently
large $j$. Set
\[
 f_j:=\frac{D_j}{\int_M\chi_j^n}>0.
\]
By \cite[Theorem~1.14]{Chen21}, there is a K\"ahler form
$\widehat\omega_j\in\alpha_j$ satisfying
\[
 Q_{\chi_j}(\widehat\omega_j)
 +f_j\frac{\chi_j^n}{\widehat\omega_j^n}=c.
\]
Thus $Q_{\chi_j}(\widehat\omega_j)<c$. Replacing $\omega_j$ by
$\widehat\omega_j$, we reduce the construction to $F=Q$.

Let $M_1$ and $M_2$ be two copies of $M$, with points denoted by
$x\in M_1$ and $y\in M_2$, and let $\pi_i:M_1\times M_2\to M_i$ be the
projections. On $M_1\times M_2$, set
\[
 \boldsymbol\alpha_j=\pi_1^*\alpha_j+\pi_2^*[\gamma],
 \qquad
 \boldsymbol\chi_j=\pi_1^*\chi_j+\pi_2^*\gamma,
 \qquad
 \boldsymbol c=c+n.
\]
The product form $\pi_1^*\omega_j+\pi_2^*\gamma$ satisfies
$Q_{\boldsymbol\chi_j}<\boldsymbol c$. The intersection number
\begin{equation}\label{eq:product-defect}
 \boldsymbol c\,\boldsymbol\alpha_j^{2n}
 -2n[\boldsymbol\chi_j]\cdot\boldsymbol\alpha_j^{2n-1}
 =\binom{2n}{n}D_j.
\end{equation}

Let $\boldsymbol\gamma_t\in
[\pi_1^*\gamma+\pi_2^*\gamma]$ be the diagonal-concentrating family of
\cite[Lemma~2.1]{DemaillyPaun}. Since $D_j\to D>0$ and
$\int_{M_1\times M_2}\boldsymbol\gamma_t^{2n}$ is independent of $t$,
this choice can be made uniformly and explicitly.  Put
\[
 \Gamma_0:=\pi_1^*\gamma+\pi_2^*\gamma,
 \qquad V_0:=\int_{M_1\times M_2}\Gamma_0^{2n}.
\]
After discarding finitely many $j$, choose $q>0$ such that
$qV_0<\binom{2n}{n}D_j$ and set
\[
 \ell_j:=\frac{\binom{2n}{n}D_j-qV_0}{V_0}>0,
 \qquad
 \rho_{j,t}:=q\boldsymbol\gamma_t^{2n}+\ell_j\Gamma_0^{2n}.
\]
Then
\[
 \rho_{j,t}\ge q\boldsymbol\gamma_t^{2n},\qquad
 \int_{M_1\times M_2}\rho_{j,t}=\binom{2n}{n}D_j.
\]
The second identity and \eqref{eq:product-defect} are the compatibility
condition for the twisted equation. Hence
\cite[Theorem~1.14]{Chen21} gives
$\boldsymbol\omega_{j,t}\in\boldsymbol\alpha_j$ satisfying
\begin{equation}\label{eq:appendix-product-equation}
 Q_{\boldsymbol\chi_j}(\boldsymbol\omega_{j,t})
 +\frac{\rho_{j,t}}{\boldsymbol\omega_{j,t}^{2n}}
 =\boldsymbol c.
\end{equation}
In particular,
\begin{equation}\label{eq:product-volume-lower}
 \boldsymbol\omega_{j,t}^{2n}
 \ge\frac{q}{\boldsymbol c}\boldsymbol\gamma_t^{2n}.
\end{equation}

Define
\begin{equation}\label{eq:fiber-current}
 T_{j,t}:=\frac1n(\pi_1)_*
 \bigl(\boldsymbol\omega_{j,t}^n\wedge\pi_2^*\gamma\bigr).
\end{equation}
This is a smooth positive $(1,1)$-form on $M_1$. Expanding its class
and using $\int_{M_2}\gamma^n=1$ gives $[T_{j,t}]=\alpha_j$.

We next prove the estimate retained by fiber integration. At
$(x,y)\in M_1\times M_2$, write
\[
 \boldsymbol\omega_{j,t}(x,y)
 =\begin{pmatrix}A&C\\ C^*&V\end{pmatrix},
 \qquad
 \mathcal S:=A-CV^{-1}C^*,
\]
relative to
$T^{1,0}_xM_1\oplus T^{1,0}_yM_2$. The block-inverse formula and
\eqref{eq:appendix-product-equation} give
\begin{equation}\label{eq:schur-Q}
 Q_{\chi_j}(\mathcal S)+Q_\gamma(V)\le c+n.
\end{equation}
For fixed $x\in M_1$, define a measure on $M_2$ by
\[
 d\mu_{j,t,x}(y):=V(x,y)^{n-1}\wedge\gamma(y)
 =\frac1nQ_\gamma(V(x,y))V(x,y)^n.
\]
Since $V(x,\cdot)$ represents $[\gamma]$ on $M_2$,
\begin{equation}\label{eq:fiber-normalizations}
 \int_{M_2}d\mu_{j,t,x}=1,\qquad
 \int_{M_2}V^n=1,\qquad
 \int_{M_2}Q_\gamma(V)V^n=n.
\end{equation}
Put $\tau:=Q_\gamma(V)$.  In a $\gamma(y)$-unitary vertical frame, the
wedge expansion in the Schur-complement calculation of
\cite[Lemmas~10.7 and~10.8]{FangMa24} gives the horizontal form
\[
 \mathcal H_{j,t}:=\mathcal S+\tau^{-1}CV^{-2}C^*\ge\mathcal S.
\]
Comparing the horizontal coefficient in
$\boldsymbol\omega_{j,t}^n\wedge\pi_2^*\gamma$ gives the exact identity
\begin{equation}\label{eq:exact-fiber-average}
 T_{j,t}(x)=\int_{M_2}\mathcal H_{j,t}(x,y)\,d\mu_{j,t,x}(y),
\end{equation}
Because $Q_{\chi_j}$ is decreasing and convex,
\eqref{eq:schur-Q} implies
\begin{align}
 Q_{\chi_j}(T_{j,t})(x)
&\le\int_{M_2}Q_{\chi_j}(\mathcal H_{j,t})\,d\mu_{j,t,x}\notag\\
&\le\int_{M_2}Q_{\chi_j}(\mathcal S)\,d\mu_{j,t,x}\notag\\
 &\le c+n-\frac1n
 \int_{M_2}Q_\gamma(V)^2V^n.\label{eq:fiber-Q-bound}
\end{align}
By Cauchy--Schwarz and \eqref{eq:fiber-normalizations},
\[
 \frac1n\int_{M_2}Q_\gamma(V)^2V^n
 \ge\frac1n
 \frac{\left(\int_{M_2}Q_\gamma(V)V^n\right)^2}
      {\int_{M_2}V^n}
 =n.
\]
Thus
\begin{equation}\label{eq:fiber-Q-final}
 Q_{\chi_j}(T_{j,t})\le c.
\end{equation}

Fix $j$ and choose $t_{j,i}\downarrow0$. The lower bound
\eqref{eq:product-volume-lower} and
\cite[Proposition~2.6]{DemaillyPaun} give, after passing to a subsequence,
\[
 \boldsymbol\omega_{j,t_{j,i}}^n\rightharpoonup\Theta_j,
 \qquad \Theta_j\ge b[\Delta],
 \qquad
 \boldsymbol\omega_{j,t_{j,i}}^{n-1}\rightharpoonup\Xi_j.
\]
The constant $b>0$ is independent of all sufficiently large $j$ because
$q$ in \eqref{eq:product-volume-lower} is uniform and the
cohomological masses and background comparisons in
\cite[Proposition~2.6]{DemaillyPaun} are uniform when
$\alpha_j\to\alpha$. Since $\Xi_j$ has
bidegree $(n-1,n-1)$ and $\Delta$ has codimension $n$, the dimension
principle gives $\mathbf 1_\Delta\Xi_j=0$.

We next extract the diagonal mass while retaining the $Q$-bound.  This
is the fixed-metric version of the localization in
\cite[Lemmas~5.6--5.10 and the proof of Proposition~5.1]{Song20}; we
include the argument because the current used here is not the modified
current appearing in that proposition.  Fix $K>4$ and choose $N>0$ so
that
\begin{equation}\label{eq:uniform-upper-reference}
 \chi_j\le N\gamma
\end{equation}
for all sufficiently large $j$.

Let $\zeta_\delta$ be a smooth function on $M_1\times M_2$, equal to
one near $\Delta$ and supported in its $\delta$-neighborhood.  Define
the near-diagonal part and its bounded replacement by
\begin{align*}
 D_{j,t,\delta}
 &:=\frac1n(\pi_1)_*\bigl(
   \zeta_\delta\boldsymbol\omega_{j,t}^n
   \wedge\pi_2^*\gamma\bigr),\\
 E_{j,t,\delta}
 &:=\frac1n(\pi_1)_*\bigl(
   \zeta_\delta\pi_1^*\chi_j\wedge
   V^{n-1}\wedge\pi_2^*\gamma\bigr),\\
 T^\circ_{j,t,\delta}
 &:=T_{j,t}-D_{j,t,\delta}+E_{j,t,\delta}.
\end{align*}
These localized forms need not be closed; throughout this paragraph,
the superscript $(r)$ denotes coefficientwise convolution in the chosen
coordinate chart.
All three forms are positive: the first two are clear, while
\eqref{eq:exact-fiber-average} gives
\begin{equation}\label{eq:localized-fiber-average}
 T^\circ_{j,t,\delta}(x)
 =\int_{M_2}\left((1-\zeta_\delta)\mathcal H_{j,t}
       +\frac{\zeta_\delta}{n}\chi_j\right)d\mu_{j,t,x}.
\end{equation}
For a constant positive form $\chi_0\le\chi_j$ on a coordinate ball,
$Q_{\chi_0}(\chi_j/n)\le n^2$.  Hence convexity,
\eqref{eq:fiber-Q-bound}, and \eqref{eq:localized-fiber-average} give
\begin{equation}\label{eq:localized-Q-bound}
 Q_{\chi_0}(T^\circ_{j,t,\delta})(x)
 \le c+n^2m_{j,t,\delta}(x),
 \qquad
 m_{j,t,\delta}(x):=
 \int_{M_2}\zeta_\delta(x,y)\,d\mu_{j,t,x}(y).
\end{equation}

We record the two consequences of the weak limits above.  First,
for every fixed convolution radius $r>0$,
\begin{equation}\label{eq:diagonal-part-lower}
 D_{j,t_{j,i},\delta}^{(r)}
 \longrightarrow
 \left(\frac1n(\pi_1)_*
  (\zeta_\delta\Theta_j\wedge\pi_2^*\gamma)\right)^{(r)}
 \ge \frac bn\gamma^{(r)}.
\end{equation}
Second, the measures $m_{j,t_{j,i},\delta}\gamma^n$ are bounded by a
fixed multiple of
\[
 (\pi_1)_*\bigl(
 \zeta_\delta\boldsymbol\omega_{j,t_{j,i}}^{n-1}
 \wedge\pi_1^*\gamma^n\wedge\pi_2^*\gamma\bigr).
\]
Every weak limit of the latter is obtained from
$\zeta_\delta\Xi_j$ by the same wedge and pushforward.  Since
$\mathbf1_\Delta\Xi_j=0$, its mass tends to zero as
$\delta\downarrow0$.  After convolution with a fixed kernel, this gives
\begin{equation}\label{eq:replacement-small}
 (m_{j,t_{j,i},\delta})^{(r)}\longrightarrow0,
 \qquad E_{j,t_{j,i},\delta}^{(r)}\longrightarrow0
\end{equation}
as first $i\to\infty$ and then $\delta\downarrow0$.  The convergence is
uniform on the smaller coordinate balls.

A finite coordinate cover, \eqref{eq:diagonal-part-lower}, and
\eqref{eq:replacement-small} now give the required uniform reserve.
Choose $\sigma>0$ so that
$(K+2)\sigma<b/(2n)$, and then choose $r_0>0$ so that
$\gamma^{(r)}$ is uniformly positive on the smaller coordinate balls
for $0<r<r_0$.  These choices are independent of all sufficiently
large $j$ and have the following property.
For every fixed $r\in(0,r_0)$ and $\varepsilon>0$, first choose
$\delta=\delta(j,r,\varepsilon)>0$ and then take
$i\ge i_0(j,r,\varepsilon,\delta)$.  Only $\sigma$ and $r_0$ are
uniform in $j$, which is all that will be used below.  One then has
\begin{align}
 Q_{\chi_0}\bigl((T^\circ_{j,t_{j,i},\delta})^{(r)}\bigr)
 &\le c+\varepsilon,
 \label{eq:localized-cone-estimate}\\
 D_{j,t_{j,i},\delta}^{(r)}
 &\ge(K+2)\sigma\gamma^{(r)},
 &E_{j,t_{j,i},\delta}^{(r)}
 &\le\sigma\gamma^{(r)}.
 \label{eq:localized-reserve}
\end{align}
Here \eqref{eq:localized-cone-estimate} follows by convolving
\eqref{eq:localized-Q-bound} and using convexity; it holds for every
admissible constant $\chi_0\le\chi_j$.

Since
$T_{j,t}=T^\circ_{j,t,\delta}+D_{j,t,\delta}-E_{j,t,\delta}$,
\eqref{eq:localized-reserve} implies
\begin{equation}\label{eq:extracted-reserve}
 (T_{j,t_{j,i}}-K\sigma\gamma)^{(r)}\ge0,
 \qquad
 (T_{j,t_{j,i}}-2\sigma\gamma)^{(r)}
 \ge (T^\circ_{j,t_{j,i},\delta})^{(r)}.
\end{equation}
By monotonicity and \eqref{eq:localized-cone-estimate},
\begin{equation}\label{eq:extracted-Q}
 Q_{\chi_0}\bigl(
 (T_{j,t_{j,i}}-2\sigma\gamma)^{(r)}\bigr)
 \le c+\varepsilon.
\end{equation}

Passing to a further subsequence, let
$T_{j,t_{j,i}}\rightharpoonup T_j\in\alpha_j$.  Convolution with a
fixed kernel is continuous under weak convergence.  Letting
$i\to\infty$ and then $\varepsilon\downarrow0$ in
\eqref{eq:extracted-reserve}--\eqref{eq:extracted-Q} gives
\begin{equation}\label{eq:limit-reserve}
 T_j\ge K\sigma\gamma,
 \qquad
 Q_{\chi_j}(T_j-2\sigma\gamma)\le c
\end{equation}
in the local-convolution sense.  In the second assertion we used that
\eqref{eq:extracted-Q} holds for every constant form
$\chi_0\le\chi_j$.

It remains to make the inequality strict and return to the limiting
class.  Put $R_j:=T_j-2\sigma\gamma$.  By
\eqref{eq:limit-reserve}, $R_j\ge(K-2)\sigma\gamma$, and
\eqref{eq:uniform-upper-reference} gives
\[
 R_j+\frac{2\sigma}{N}\chi_j\le T_j.
\]
Apply \cref{lem:elementary} to the absolutely continuous parts and then
use \cref{lem:ac-cone}.  We obtain
\begin{equation}\label{eq:strict-Q-before-class-limit}
 Q_{\chi_j}(T_j)
 \le \frac{c}{1+2\sigma c/(nN)}=:c_0<c,
\end{equation}
where $c_0$ is independent of all sufficiently large $j$.

Choose smooth closed forms $h_j\in\alpha-\alpha_j$ with
$\lVert h_j\rVert_{C^\infty}\to0$, for instance the harmonic
representatives with respect to $\gamma$.  For a sufficiently large fixed
$j$, put
\[
 \theta_j:=\frac{\lVert h_j\rVert_\gamma}{K\sigma}.
\]
Then $-\theta_jT_j\le h_j\le\theta_jT_j$ and $\theta_j\to0$.  Thus the
current
\[
 T:=T_j+h_j\in\alpha
\]
satisfies $T\ge(1-\theta_j)T_j$.  Choose $j$ so large that
$\theta_j\le1/2$ and
$(1-\theta_j)^{-1}c_0\le(c+c_0)/2$.  Monotonicity and homogeneity now
give, on the absolutely continuous parts,
\[
 T\ge \frac{K\sigma}{2}\gamma,\qquad
 Q_{\chi_j}(T)
 \le(1-\theta_j)^{-1}Q_{\chi_j}(T_j)
 \le\frac{c+c_0}{2}=c-\eta,
 \qquad \eta:=\frac{c-c_0}{2}>0,
\]
and \cref{lem:ac-cone} gives the local-convolution statement. Finally take
$\lambda:=\delta_j$, so that $\chi+\lambda\gamma=\chi_j$, and set
$a:=K\sigma/2$.  This proves \textup{(i)} for $F=Q$; the case $F=P$
follows from $P\le Q$.
\end{proof}